\documentclass[4apaper]{amsart}
\usepackage[utf8]{inputenc}

\usepackage{enumitem}

\usepackage[top = 3cm, bottom = 4cm, left = 3cm, right = 3cm]{geometry}
\usepackage[english]{babel}
\usepackage{amsmath, amssymb, latexsym, amsfonts, enumitem} 
\usepackage[matrix,arrow,curve]{xy}
\usepackage[usenames]{color}
\usepackage{colortbl}
\usepackage{xcolor}

\usepackage{amsthm}

\numberwithin{equation}{section}
\theoremstyle{plain}
\newtheorem{theorem}[equation]{Theorem}

\newtheorem{corollary}[equation]{Corollary}
\newtheorem{proposition}[equation]{Proposition}
\newtheorem{lemma}[equation]{Lemma}

\theoremstyle{definition}

\newtheorem{definition}[equation]{Definition}

\newtheorem{example}[equation]{Example}

\theoremstyle{remark}
\newtheorem{remark}[equation]{Remark}

\renewcommand{\leq}{\leqslant}
\renewcommand{\geq}{\geqslant}

\newcommand{\Cor}{\mathrm{Cor}}
\newcommand{\DM}{\mathbf{DM}}
\newcommand{\Pic}{\operatorname{Pic}}
\newcommand{\pr}{\mathrm{pr}}

\newcommand{\calI}{\mathcal I}
\newcommand{\can}{\mathrm{can}}
\newcommand{\ovX}{\overline{X}}
\newcommand{\wtovX}{\overline{X}}
\newcommand{\calX}{\mathcal X}
\newcommand{\calE}{\mathcal E}
\newcommand{\ovcalX}{\overline{\mathcal X}}
\newcommand{\bbZ}{\mathbb Z}
\newcommand{\Sing}{\mathrm{Sing}}
\newcommand{\Iso}{\mathrm{Iso}}
\newcommand{\calC}{\mathcal C}
\newcommand{\Hom}{\mathrm{Hom}}
\newcommand{\Aut}{\mathrm{Aut}}
\newcommand{\id}{\mathrm{id}}
\newcommand{\Jac}{\operatorname{Jac}}
\newcommand{\Id}{\mathrm{Id}}

\renewcommand{\leq}{\leqslant}
\renewcommand{\geq}{\geqslant}

\newcommand{\GL}{\mathrm{GL}}
\newcommand{\SL}{\mathrm{SL}}
\newcommand{\El}{\mathrm{E}}

\newcommand{\Kth}{\mathrm{K}}
\newcommand{\GW}{\mathrm{GW}}
\newcommand{\Witt}{\mathrm{W}}

\newcommand{\struct}{\mathcal{O}}
\newcommand{\relcurve}{\mathcal{C}}
\newcommand{\relpcurve}{\overline{\mathcal{C}}}
\newcommand{\Spec}{\operatorname{Spec}}
\newcommand{\A}{\mathbb A}
\newcommand{\PP}{\mathbb P}

\newcommand{\SH}{\mathbf{SH}}

\newcommand{\codim}{\operatorname{codim}}
\newcommand{\rank}{\operatorname{rank}}
\newcommand{\chark}{\operatorname{char}}

\newcommand{\etale}{\'etale }

\newcommand{\pour}{\ar@{}[ur]|(0.2){\text{\pigpenfont G}}}
\newcommand{\podr}{\ar@{}[dr]|(0.2){\text{\pigpenfont A}}}
\newcommand{\et}{\mathrm{\acute{e}t}}
\newcommand{\op}{\mathrm{op}}
 
\newcommand{\lF}{\mathrm{F}}
\newcommand{\Fr}{\mathrm{Fr}}

\newcommand{\Sm}{\mathrm{Sm}}
\newcommand{\Sch}{\mathrm{Sch}}
\newcommand{\SmAff}{\mathrm{SmAff}}
\newcommand{\ZF}{\mathbb Z\lF}
\newcommand{\ZFr}{\mathbb Z\Fr}
\newcommand{\ovZF}{\overline{\ZF}}
\newcommand{\ovZFr}{\overline{\ZFr}}

\newcommand{\calO}{\mathcal O}
\newcommand{\pt}{\mathrm{pt}}

\newcommand{\Zar}{\mathrm{Zar}}

\newcommand{\nr}{\mathrm{nr}}

\newcommand{\ovrelcurve}{\overline{\relcurve}}
\newcommand{\ovC}{\overline C}
\newcommand{\calL}{\mathcal L}
\newcommand{\calR}{\mathcal R}

\newcommand{\Gm}{\mathbb G_m}
\newcommand{\Ker}{\operatorname{Ker}}
\newcommand{\Coker}{\operatorname{Coker}}
\newcommand{\Ab}{\mathrm{Ab}}

\newcommand{\Et}{\mathrm{Et}}

\begin{document}

\title{
Rigidity for smooth affine henselian pairs over a field}
\thanks{The research is supported by the Russian Science Foundation grant 19-71-30002}
\author{A.~Druzhinin}

\address{Chebyshev Laboratory, St. Petersburg State University, 14th Line V.O., 29B, Saint Petersburg 199178 Russia}
\email{andrei.druzh@gmail.com}

\begin{abstract}
Let $Z\to X$ be a closed immersion of smooth affine schemes over an arbitrary field $k$,
and $X^h_Z$ denote the henselization of $X$ along $Z$.
For each presheaf 
$E\colon \mathbf{SH}(k)\to \mathrm{Ab}^\mathrm{op}$
on the stable motivic homotopy category over $k$
and the induced continuous presheaf  
$E\colon \mathrm{EssSm}_k\to \mathrm{Ab}^\mathrm{op}$
on the category of essentially smooth schemes
there is a homomorphism \[E(X^h_Z)\to E(Z).\]
We prove that this is an isomorphism 
for any $l_\varepsilon$-torsion presheaf $E$, 
for $l\in \mathbb Z$, $(l,\chark k)=1$, and $l_\varepsilon=\sum_{i=1}^n \langle (-1)^i \rangle$.
More generally, the isomorphism holds for any 
homotopy invariant $l_\varepsilon$-torsion linear $\sigma$-stable framed additive presheaf $F$ over $k$.
The case of $l$-torsion presheaves follows as well.

The result generalises 
known Gabber's rigidity theorems
for local henselian schemes
to the case of smooth affine henselian pairs.

The above isomorphism is proven by 
constructing 
of 
(stable) $\A^1$-homotopies of motivic spaces via algebro-geometric techniques.
To achieve this in our setting 
we replace often used Quillen's trick
by an alternative construction that provides required smooth relative curves 
over smooth affine schemes for an arbitrary base field. 

\end{abstract}

\maketitle

\section{Introduction}

Consider an immersion of (smooth) (compact) manifolds $Z\to X$, and 
the $\varepsilon$-neighbourhood 
\begin{equation}\label{eq:espliontubetopological}X_{Z,d,\varepsilon}=\{x\in X| d(x,Y)<\varepsilon\},\end{equation} 
with respect to some distance $d\colon X\times X\to X$ agreed with the topology on $X$. 
Then $Z$ is a retract of $X_{Z,d,\varepsilon}$ for all small enough $\varepsilon>0$. 
Consequently, for any homotopy invariant presheaf $F$ 
there is the isomorphism $F(Y)\simeq F(X_{Z,d,\varepsilon}).$
In particular, the isomorphism holds for any 
cohomology theory on the category of manifolds (or appropriate topological spaces).
We can write
\begin{equation}\label{eq:topinftuberigid}F(Z)\simeq \varinjlim_{\varepsilon\to +0} F(X_{Z,d,\varepsilon}) ,\end{equation}
then the right side is independent of the choice of the distance $d$ (and $\varepsilon$). 

In the article, we study analogue of the above isomorphism for motivic (generalised) cohomology theories defined on 
the Voevodsky-Morel stable motivic homotopy category over a field $k$ \cite{MV99,Jar00,Mor0}. 
The motivic homotopy category 
$\mathbf H(k)$ and the stable category $\SH(k)$
are considered as analogues to the classical (topological) ones in the 
algebraically geometric setting in view of the following: 
(1) the role of homotopies, 
belongs to so-called $\A^1$-homotopies; 
(2) the role of the topology 
is played by the Nisnevich topology on $\Sm_k$, 
(3) the spheres are the simplicial sphere $S^1$ and the motivic spheres: 
$\mathbb G_m^{\wedge 1}=\Gm/\{1\}$,
and $T\simeq \PP^1/\{\infty\}\simeq \Gm^{\wedge 1}\wedge S^1$. 

\subsection{\'Etale neighbourhoods and Gabber's rigidity}\label{subsect:EtaleNeighGabRig}
One of the geometrical features 
that effect for the choice of the Nisnevich topology 
in the construction of the motivic homotopy category is that 
the topology has a natural notion of 
neighbourhoods of closed subschemes, 
that look like $\varepsilon$-tubes \eqref{eq:espliontubetopological}. 
Namely, these are called \'etale neighbourhoods. 
At the same, time this allows 
to ask
by analogy with \eqref{eq:topinftuberigid} about the isomorphism
\begin{equation}\label{eq:infNisnNeighRigidityisomorphism}E(Z)\simeq \varinjlim_{Z\to V\to X} E(V),\end{equation}
where $V$ runes over the filtering system of the \'etale-neighbourhoods of $Z$ in $X$.
If $F$ is a continuous presheaf on the category of 
schemes, then 
the injective limit in \eqref{eq:infNisnNeighRigidityisomorphism} equals to $F(X^h_Z)$, 
where $X^h_Z=\varprojlim_{V} V$ in the category of schemes. 

The scheme $X^h_Z$ is usually considered in the literature in the affine case.
If $X=\Spec A$ and $Z=\Spec A/I$, for a ring $A$ and an ideal $I\subset A$,
then $X^h_Z=\Spec A^h_I$,
where the henselization $A^h_I$ of the ring $A$ with respect to $I$
is the subring of the $I$-adical completion $A^\wedge_I$ that contains algebraic elements over $A$ that are 
separable over $A/I$.
Then \eqref{eq:infNisnNeighRigidityisomorphism} can be rewritten as
\begin{equation}\label{eq:affineAYXBrigidityisomorphism}F(\Spec A/I)=F(Z)\simeq F(X^h_Z)=F(\Spec A^h_I) .\end{equation}
The pair $(A,I)$ such that $A=A^h_I$ is called \emph{an affine henselian pair or couple}.
One probable reason why the affine henselian pairs are considered much more often 
is that for some examples of projective schemes 
the  \'etale neighbourhoods could become equal to Zariski open neighbourhoods, that are distinctly bigger in the affine case. 

Isomorphism \eqref{eq:affineAYXBrigidityisomorphism}
is called Gabber's rigidity isomorphism 
after the work \cite{Gab92},
where isomorphism \eqref{eq:affineAYXBrigidityisomorphism} 
is proven for the algebraic $\Kth$-theory with finite coefficients 
$\bbZ/l\bbZ$, for an integer $l$ invertible in the (noetherian) ring $A$. 
\cite{Gab92} elaborates the method of Suslin's pioneer works \cite{Sus83,Sus84} 
that contain in particular the proof of the isomorphism for the case of discrete valuation ring.

In the article, we 
consider 
the case of smooth affine schemes $Z,X\in \SmAff_k$, and
presheaves $F\colon \SmAff_k\to \mathrm{Ab}^\mathrm{op}$ over a field $k$.
The immersion $Z\to X^h_Z$ we call \emph{smooth affine henselian pair} over $k$.

The basic case of Gabber's rigidity 
is the case $Z=z$ being a closed point of $X$.
Then $X^h_Z$ is a local henselian scheme, 
that is analogue of
the infinitesimally small disk around the point
in topological terms.
The case of (smooth) local henselian pairs over a field 
was widely studied in literature for different types of presheaves,
see Remark \ref{rem:LocCaseoTheoremsinArticle}, 
and Subsection \ref{subsect:Gabbersrigidityreview}.
In the non-local affine case isomorphism \eqref{eq:affineAYXBrigidityisomorphism} is proven only for few examples of cohomology theories till now.

\subsection{The result.} 
The motivic homotopy theory is 
less flexible, then the classical topological one. 
It is because the condition to be a regular map of algebraic varieties is much stronger then for continuous maps of 
manifolds.
Nevertheless, it demonstrates similar behaviour from some viewpoints.
The isomorphism \eqref{eq:affineAYXBrigidityisomorphism} does not hold for 
a general $\A^1$-homotopy invariant presheaf $F\colon \SmAff_k\to \Ab^\mathrm{op}$, 
or a general additive functor 
$E\colon \SH(k)\to \Ab^\mathrm{op}$,
but it holds under stronger assumptions.

An enough additional assumption is that $E$ is torsion with respect to some element 
$\phi$ in $[\mathbb S,\mathbb S]_{\SH(\mathbb Z)}$ or $[\mathbb S,\mathbb S]_{\SH(k)}$
such that 
$\rank(\phi)\in \mathbb Z$ is invertible on $Z$.
Here $\mathbb S$ denotes the motivic sphere spectrum and
$\rank(-)$ denotes the canonical ring homomorphism
$[\mathbb S,\mathbb S]_{\SH(k)}\to \mathbb Z$ 
induced by the functor 
$\SH(k)\to \mathbf{DM}(k)$ to the category of Voevodsky's motives \cite{VSF00}. 
By Morel's theorem on zero-th stable motivic homotopy groups, see \cite{Mor04}, and \cite[remark 26]{Mor12}, 
there is the isomorphism 
$[\mathbb S,\mathbb S]_{\SH(k)}\simeq \mathrm{GW}(k)$, 
where $\mathrm{GW}(k)$ is the Grothendieck-Witt ring.
In view of this $\rank(\phi)$ is the rank of a quadratic space $\phi\in \mathrm{GW}(k)$.

The classes of $n$-torsion presheaves, for $n\in \mathbb Z$, 
and $n_\varepsilon$-torsion presheaves 
are contained in the class of $\phi$-torsion presheaves defined above.
Here $n_\varepsilon=\sum_{i=1}^n \langle (-1)^i \rangle $,
and 
the element $\langle \lambda \rangle\in [\mathbb S,\mathbb S]_{\SH(k)}$, for $\lambda\in k^\times$, 
is given by the morphism $\Gm^{\wedge 1}\to \Gm^{\wedge 1}$ induced by the map $x\mapsto \lambda x$. 
In particular, the motivic cohomologies with $\mathbb Z/n\mathbb Z$ coefficients are $n$-torsion;
the (derived) Witt-groups presheaves $\mathrm{W}^i(-)$, $i\in \mathbb Z/4\mathbb Z$, 
are torsion with respect to the element $h=2_\varepsilon=\langle 1\rangle +\langle -1\rangle$.

Consider the presheaf 
\begin{equation}\label{eq:Sm->SH-E>Ab}
\Sm_k\to \SH(S)\xrightarrow{E} \Ab^\mathrm{op}
\end{equation}
given by the additive functor $E$. 
Denote by the same symbol 
the corresponding continuous presheaf 
on essentially smooth schemes.

\begin{theorem}\label{th:introduction:Rigidity_E:SH->Ab_phi}
Let $Z\to X$ be a closed immersion of smooth affine schemes over a field $k$.
Let $\phi\in [\mathbb S,\mathbb S]_{\SH(k)}$ such that
$\rank (\phi)\in\mathbb Z$ defines an invertible element in $k$.

Then for any $\phi$-torsion additive functor $E$, see \eqref{eq:Sm->SH-E>Ab}, the inverse image homomorphism induces the isomorphism
\begin{equation}\label{eq:i*E(XhZ)simeqE(X)}i^*\colon E(X^h_Z)\to E(Z),\end{equation}
where $X^h_Z$ is the henselization of $X$ along $Z$, $i\colon Z\to X^h_Z$ is the canonical closed the immersion.
\end{theorem}


\begin{remark}\label{rem:EDM->AbSmAffHenseCor->Ab}
Theorems \ref{th:introduction:Rigidity_E:SH->Ab_phi}, and \ref{th:introduction:Rigidity:framed} 
are new in the case of 
$n$-torsion homotopy invariant presheaves with transfers $\mathrm{Cor}\to \Ab^\mathrm{op}$
or (torsion) presheaves 
on the category of Voevodsky's motives $\mathbf{DM}(k)$. 
\end{remark}

\begin{remark}\label{rem:LocCaseoTheoremsinArticle}
The local (henselian) case of isomorphism \eqref{eq:i*E(XhZ)simeqE(X)} 
for (torsion) presheaves with transfers over a field $k$ was proven in \cite{SV96}.
The local case of Theorems \ref{th:introduction:Rigidity_E:SH->Ab_phi} and \ref{th:introduction:Rigidity:framed} is proven in \cite{AD-Rigid}.
Namely, the local case is the case of $Z=z$ being a closed 
point in the scheme $X\in \Sm_k$ 
with the separable residue field $k(z)/k$. 

The main novelty in the article are geometrical 
constructions needed for non-local smooth affine case,
described in Subsubsections \ref{subsubsect:KeyConstructCurve} and \ref{subsubsection:A1hommotsuspspaces}. 
The rigidity isomorphism is proven via 
some kind of a 'curve'-homotopy that retracts the affine scheme 
$X^h_Z$ to 
$Z$ inside $X$.
The 'curve'-homotopy leads to 
the 
(stable) $\A^1$-homotopy on the suspension motivic spectra 
up to an $n_\varepsilon$-divisible morphism in $\SH(k)$.
\end{remark}




\begin{remark}

The symbol $X^h_Z$ sometimes is used with different meaning.
In the recent preprint \cite{JPAdddoublecos} by Joshua and Pelaez 
the notation $X^h_Z$ is used with the reference to \cite{Cox78} by Cox, where the notion of \emph{henselian schemes} is introduced.
In distinct to the henselizations $X^h_Z$ 
the \emph{henselian schemes} in \cite{Cox78} 
are not schemes (or pro-schemes), but are rigid spaces of another type.

The preprint \cite{JPAdddoublecos} contains an argument that deduces rigidity on objects denoted by $X^h_Z$ from 
the local rigidity property.
The argument is based on the local to global spectral sequence, 
and the used properties of $X^h_Z$ does not hold for the affine schemes 
$X^h_Z=\Spec A^h_I$ over a field $k$. 

\end{remark}

\begin{remark}
The case of $n_\varepsilon$-torsion presheaves
in Theorem \ref{th:introduction:Rigidity_E:SH->Ab_phi} is equivalent to the general case of 
$\phi$-torsion presheaves. 

The torsion condition with respect to an element $\phi\in [\mathbb S,\mathbb S]_{\SH(k)}$ 
for rigidity isomorphisms 
was suggested by A.~Ananiesky.
The first version of \cite{AD-Rigid} concerned $n$-torsion and $n_\varepsilon$-torsion presheaves.
The statement in the general form for $\phi$-torsion presheaves and the short argument for the equivalence with the $n_\varepsilon$-case ware found by T.~Bachamnn and A.~Ananievskiy. 
\end{remark}

\subsection{Elemntary form of the result.}\label{subsect:RigidFr} 

As was said before in the topological case 
isomorphsim \eqref{eq:topinftuberigid} holds 
for any homotopy invariant presheaf $F$,
such properties of a cohomology theory as the Mayer-Vietoris exact sequence are not needed.

There is a stronger and more elementary version of 
Theorem \ref{th:introduction:Rigidity_E:SH->Ab_phi}.
It does not formally use 
the category $\SH(k)$,
and the Brown-Gersten property. 
This can be done in terms of so called framed correspondences 
introduced by Voevodsky \cite{Voev-FramedCorrM}, and more wiedly studied 
in the Garkusha-Panin's theory of framed motives 
\cite{GP14}. 
At the same time this is only one known way at the moment 
to prove Theorem \ref{th:introduction:Rigidity_E:SH->Ab_phi}.


The (graded) category $\Fr_+(k)$ is 
a category with objects smooth schemes over $k$
and morphisms given by 
certain diagrams (of schemes), 
see Def. \ref{def:Fr}.
The definition implies canonical functors 
\begin{equation}\label{eq:Fr+->SH}\Sm_k\to \Fr_+(k)\to \SH(k).\end{equation} 
The category $\SH(k)$ 
can be constructed in the following steps
\begin{itemize}
\item[(1)] to consider the stable homotopy category of $S^1$-spectra presheaves on $\Sm_k$, 
\item[(2)] to make the localisation with respect to $\A^1$-homotopy (weak) equivalences, Nisnevich local (weak) equivalences, and $\mathbb G_m^{\wedge 1}$-stable weak equivalences.
\end{itemize}
In parallel, we can consider 
\begin{itemize}
\item[(1)] the additivisation $\ZF_*(k)$ of $\Fr_+(k)$,
\item[(2)] and the $\A^1$-homotopy quotient (graded) category $\ovZF_*(k)$.
\end{itemize}
Then the second functor in \eqref{eq:Fr+->SH} induces the (additive) functor
\begin{equation}\label{eq:ZF*toSH}\overline{\ZF}_*(k)\to \SH(k).\end{equation} 
%


As a paricular consequence of main results in 
\cite{GP14}
it follows that for a perfect base field $k$ 
for any $Y\in \Sm_k$ the morphism of presheaves
\begin{equation}\label{eq:ovZF(Y)pi-(Y)}\varinjlim_n \ovZF_n(-,Y)\to [-,Y]_{\SH(k)},\end{equation}
induced by \eqref{eq:ZF*toSH} 
is a (Nisnevich and Zariski) local isomorphism on $\Sm_k$.
So \eqref{eq:ZF*toSH} can be considered as 
a (robust) approximation of $\SH(k)$
by an additive category with objects smooth schemes and 
hom-groups defined with generators and relations described in (precise and technical) algebro-geometric terms.



We should note that the actual aim and power of the framed correspondences approach is 
the reconstruction of $\SH(k)$ over a perfect field 
in order to get a computational tool for the stable motivic homotopy groups, 
but we do not use it in the present article.
%
Moreover, the local isomorphism 
\eqref{eq:ovZF(Y)pi-(Y)} is not needed for our result too. 
We use 
the functor \eqref{eq:ZF*toSH} well-defined for any base.
In view of \eqref{eq:ZF*toSH} any additive presheaf $E\colon \SH(k)\to \Ab^\mathrm{op}$ defines a presheaf
\[F\colon \overline{\ZF}_*(k)\to \Ab^\mathrm{op}.\]
Moreover, any such a presheaf $F$ 
satisfies the property called $\sigma$-stability, 
and if $E$ is $n_\varepsilon$-torsion, for $n\in \mathbb Z$, 
then $F$ is $n_\varepsilon$-torsion in the corresponding sense.
The functors $F$ and $E$ induce the same functor on essentially smooth schemes.
Theorem \ref{th:introduction:Rigidity_E:SH->Ab_phi} is a consequence of the following result.
\begin{theorem}\label{th:introduction:Rigidity:framed}
Suppose $i\colon Z\hookrightarrow X^h_Z$ is an affine smooth henselian pair over a field $k$
as in Theorem \ref{th:introduction:Rigidity_E:SH->Ab_phi}.
Let $F\colon \mathbb Z\mathrm{F}_*(k)\to \Ab^\mathrm{op}$ 
be an $\A^1$-homotopy invariant 
$\sigma$-stable additive functor. 
Assume that $l_\epsilon F=0$ for some $l\in \mathbb Z$, $(l,\chark k)=1$.
Then the inverse image homomorphism defines the isomorphism 
\[i^*\colon F(X^h_Z)\simeq F(Z).\]
\end{theorem}

Framed correspondences
ware suggested by Voevodsky
as a tool for $\SH(k)$
that would work in parallel to
Voevodsky's correspondences $\mathrm{Cor}(k)$ in the case of $\mathbf{DM}(k)$.
Gabber's rigidity on (essentially) smooth local henselian schemes 
for $\A^1$-homotopy invariant torsion presheaves with $\Cor(k)$-transfers \cite[Theorem~4.4]{SV96}
implies the same for the torsion additive functors $E\colon \mathbf{DM}(k)\to \Ab^\mathrm{op}$.
In \cite{AD-Rigid} the using of framed correspondences allowed to generalise \cite[Theorem~4.4]{SV96} 
with the application to $\SH(k)$ and prove Theorems \ref{th:introduction:Rigidity_E:SH->Ab_phi}, and \ref{th:introduction:Rigidity:framed} for (essentially) smooth local henselian schemes.

Our argument and result would be still new in the $\Cor$-correspondences case, see Subsection \ref{subsect:Intr:AffinepairRelCurve} for details. 
%
At the same time framed correspondences appear naturally 
in the needed algebro-geometric constructions.
The constructions use some kind of framed correspondences of positive dimension,
while $\Cor(X,Y)$ and $\Fr_n(X,Y)$ are correspondences with finite support,
see Subsection \ref{subsect:FrCorddiminner} for details. 

\subsection{Gabber's rigidity in literature} 
\label{subsect:Gabbersrigidityreview}
In few words we can say that
the rigidity isomorphisms are of interest
because they connect 
values of a given cohomology theory on different schemes. 
Sometimes this plays a role like $\A^1$-homotopy invariance. 
The rigidity is formally an independent property, but
it is often more special and restrictive condition on the cohomology theory.

In this subsection we review known (Gabber's) rigidity isomorphism.
In some works the rigidity plays the role of an instrument oriented to the main goal, 
sometimes it is considered as an application,
in the most of listed works the rigidity is the main result.
We group the works
with respect to the goal and the motivation,
and 
with respect to the methods provided the rigidity.

Some articles and preprints mentioned in this and the next subsections, namely \cite{HRO1Oct,BachmannRigidSHet,RigidityEquivariantAlgKtheoryNimanRavi}, 
appeared on arXiv later the present one;
reworking the text we decided to include them for the completeness and freshness of our review.

\subsubsection{Rigidity for $\mathrm{K}$-theory}
The rigidity theorems started with 
works by Suslin on algebraic $K$-theory \cite{Sus83,Sus84}.
Namely, the works contain the rigidity isomorphisms
for preshaves $\Kth_i(-, \mathbb Z/n\mathbb Z)$, $i\in \mathbb Z_{\geq 0}$, 
along
extensions of algebraically closed fields,
and
the isomorphism of the form \eqref{eq:affineAYXBrigidityisomorphism} 
for a closed point $Z=z$ of a Dedekind scheme $X$.
In both results $n$ is an integer invertible on corresponding schemes.
The result \cite[Theorem~A]{GT84} by Gillet and Thomason states 
the rigidity for the same presheaves
for the strict henselization $X^\mathrm{sh}_z$ at a point $z$
of a smooth scheme $X$ over a separably closed field. 
In Gabber's work \cite[Theorem~1]{Gab92} the isomorphism
\begin{equation}\label{eq:RigidKth}
\Kth_i(A,\mathbb Z/l\mathbb Z)\simeq \mathrm \Kth_i(A/I,\mathbb Z/l\mathbb Z), \forall i\geq 0.
\end{equation}
was proven for any henselian pair $(A,I)$ of $\mathbb Z[1/l]$-algebras. 
There is the review by T.~Geisser \cite{Greview00} on the mentioned above results.

The approach 
was developed in \cite{RosRigidKthquatization} by Rosenberg 
with respect to question on K-theory with finite coefficients of the quantization deformation of algebras 
(instead of henselian pairs).

The rigidity isomorphisms ware involved originally by Suslin \cite{Sus83,Sus84}
in order to compute the $\mathrm{K}$-theory 
of algebraically closed fields, and local fields, 
and to prove the Quillen-Lichtenbaum conjecture.
The rigidity isomorphisms help to transfers computational results between different fields and rings.

\subsubsection{Cyclotomic trace map}

Gabber's rigidity isomorphisms appear also 
is the studying of 
the cyclotomic trace map for the algebraic $\mathrm{K}$-theory. 
The goal of cyclotomic trace methods is the computation of $\mathrm{K}$-theory 
for some nilpotent rings like $\mathbb Z/p^l\mathbb Z$ or $p$-adically complete rings. 
The rigidity results are involved because of the same reason as in the previous paragraph. 

The most recent works in this direction we know are 
\cite{CMM18} by Clausen, Mathew, Morrow, and 
\cite{RigidityEquivariantAlgKtheoryNimanRavi} by N.~Naumann, C.~Ravi.
The central result of \cite{CMM18} is 
the isomorphism \eqref{eq:RigidKth} 
for an arbitrary $l\in \mathbb Z$ with 
the $K$-theory spectrum $K(R)$ replaced by 
the homotopy fiber of the cyclotomic trace $K(R)\to TC(R)$.
In \cite{RigidityEquivariantAlgKtheoryNimanRavi} 
the method is developed in the equivariant case,
as a consequence 
the Gebber's rigidity in the original form \eqref{eq:RigidKth}
is proven for equivarint $\mathrm{K}$-theory with respect to a (finite)
group that order is invertible in $R$.

\subsubsection{\'Etale cohomologies with torsion coefficients.}
Also one case of rigidity isomorphism for affine henselian pairs \eqref{eq:affineAYXBrigidityisomorphism} is the case 
of \'etale cohomologies with finite coefficients, 
and more generally with torsion coefficients 
on the small \'etale site on $\Et_X$. 
This is proven in \cite{Gab94} by Gabber and in \cite{HEtfin} by Hubber.
In \cite{PoppiPBchAffHensPair} there is a reference for this result to much earlier work \cite{Elkiksoleqhens} by Elkik.
Some examples of rigidity for \'etale cohomologies for non-affine henslian pairs are mentioned in Subsection \ref{subsection:ninaffinecase}.

\subsubsection{Suslin-Voevodsky's method for the rigidity}
In the work \cite{SV96} by Suslin and Voevodsky
there was presented an approach
provided the rigidity isomorphisms 
for a wide class of presheaves over a field $k$, see \cite[Theorem~4.4]{SV96}.
This is the class of homotopy invariant (torsion) presheaves with transfers, 
i.e. additive presheaves in the 
the category of correspondences $\mathrm{Cor}_k$.

The ideas and the method 
ware developed and generalised 
to cover many other classes of presheaves.
Let us list works that can be related to this direction
\cite{PY02} by Panin and Yagunov,
\cite{Ya04} by Yagunov,
\cite{HY07} by Hornbostel and Yagunov,
\cite{RO08} by R\"ondigs and \O stvaer,
\cite{N15} by Neshitov,
\cite{Ananyevskiy-zthstpiA1(smcurve)} by Ananyevskiy,
\cite{D-WittRigid} by the author,
\cite{AD-Rigid} by Ananyevskiy and the author.
In some of the references the method was developed 
not for the Gabber's rigidity for henselian pairs,
but also for the case of field extensions
(proven originally for $\mathrm{K}$-theory in \cite{Sus83}).
The work \cite{HRO1Oct} by Heller, Ravi, and \O stvaer, 
that proves the rigidity isomorphism
for some equivariant cohomology theories,
could be partly related to the direction of the discussed method as well.

The main subject of 
\cite{SV96} is the isomorphism of 
motivic homologies with finite coefficients of a smooth scheme $X$ over $\mathbb C$
and
the singular homologies with the same coefficitents
of the complex realisation $X(\mathbb C)$.
The rigidity theorem allowed 
to prove that 
the corresponding cohomology presheaves
are locally constant on $\Sm_{\mathbb C}$ with respect to \'etale topology.

The work \cite{SV96} had gave some part of the basis for 
the future development of Voevodsky's motives theory. 
It is contained 
implicitly in \cite{SV96} the proof of the equivalence
$\DM_{\text{\'et},\mathrm{eff}}(\mathbb C,\bbZ/l\bbZ)\simeq \mathbf{D}(\bbZ/l\bbZ\,\text{-mod})$.
Results of such type for \'etale motivic categories
\cite{VSF00,Ayo13,CD16,BachmannRigidSHet}
are called rigidity theorems too, see Subsection \ref{subsect:SHCatPropertyRigidity}. 
These results in its turn imply Gabber's rigidity isomorphisms for 
essentially smooth local henselian schemes.

\subsubsection{Friedlander-Milnor's conjecture}

\cite{FriedlanderFrielanderMilnorConj,MilnorFrielanderMilnorConj} for a reductive algebraic group $G$ over an algebraically closed field $k$ follows from Gabber's rigidity \eqref{eq:affineAYXBrigidityisomorphism} on one-dimensional schemes for presheaves $H^i(G(-),\mathbb Z/l\mathbb Z)$, $l\in k^\times$, see Knudson's work \cite{K98}, or Jardins's article \cite{JardineFriedMilnConjRigidityHGl}.

Suslin's rigidity theorems \cite{Sus83,Sus84} are based on the studying of cohomologies of the stable linear group $\mathrm{GL}_\infty$. The obtained results imply the Friedlander Milnor's conjecture for $\mathrm{GL}_\infty$.

In \cite{K98} the mentioned rigidity for $G=\mathrm{GL}_n$, $i\geq 3$, over an infinite base field $k$ on a henselian local $k$-algebra $A$ with the residue field $k$, that implies Friedlander-Milnor's conjecture for such $G$.

In \cite{Mor11} Morel proved 
a weak form of Friedlander-Milnor's conjecture
for any reductive algebraic group $G$ over a separably closed field $k$.
This implies the original conjecture,
in particular, if
the presheaf $H^*(G(-),\mathbb Z/l\mathbb Z)$, $l\in k^\times$, 
is $\A^1$-homotopy invariant.
%
The key ingredient 
in \cite{Mor11} is Gabber's rigidity for 
cohomologies with finite coefficients of the classifying space of the presheaf $\Sing_*^{\A^1}(G)$, 
where $\Sing_*^{\A^1}$ is the Suslin-Voevodsky's construction of the $\A^1$-localisation.
For this purpose, and basing on Suslin-Voevodsky's method \cite[Theorem~4.4]{SV96},
Morel proved rigidity theorem
for a strictly $\A^1$-invariant sheaf with generalised transfers 
\cite[Theorem~5.14]{Mor11}.

\subsubsection{$\rho$-periodic cohomology theories}
The computation of 
the $\rho$-inverted stable motivic homotopy category 
in terms of spectra on the real \'etale site
by Bachmann \cite{Ba16} 
has a rigidity result as a consequence \cite[Corollary~44]{Ba16}.
Namely, it is the Gabber's rigidity 
on essentially smooth local henselian schemes 
for  $\rho$-periodic cohomology theories.
The element $\rho$ is defined by 
$\rho=-[-1]\in \mathrm{K}^{\mathrm{MW}}_1(k)\cong[\mathbb S,\mathbb G_m^{\wedge 1}]_{\SH(k)}$.

\subsubsection{Non-local affine case}
The works by Gabber \cite{Gab92, Gab94} that gave the name to isomorphism \eqref{eq:affineAYXBrigidityisomorphism} concerned affine henselian pairs. At the same time, most of Gabber's rigidity isomorphisms proven in literature cover the local henselian case.
The non-local affine case 
is more complicated even for particular examples of presheaves. 
In the above list the non-local rigidity results are 
\cite{Gab92,Gab94, CMM18, RigidityEquivariantAlgKtheoryNimanRavi}.

In particular,  the generality of discussed above Suslin-Voevodsky's method used originally in \cite[Theorem 4.4]{SV96}
is the case of the local henselian scheme $X^h_x$ with $X\in \Sm_k$, and $x\in X$ being a point with a separable residue field over a field $k$.
Our work evolves this approach to the rigidity isomorphisms, though the basic constructions differ from the beginning. 
The common features are using of transfers and studying of relative curves and their compactifications, see Subsection \ref{subsect:Intr:AffinepairRelCurve} \ref{subsect:FrCorddiminner} for the essential differences.

\subsubsection{Non-affine case}\label{subsection:ninaffinecase}

In \cite{Gab94} Gabber called the isomorphism \eqref{eq:affineAYXBrigidityisomorphism} as an affine analogue of the proper base change.
In its turn, the proper base change for \'etale cohomologies with finite coefficients over henselian pairs gives an example of
rigidity isomorphism \eqref{eq:infNisnNeighRigidityisomorphism}
for certain non-affine schemes. 
See \cite[ch~4, Corollary~2.7]{Milne-EtCoh} for the (classical) case over local henselian pair, and the case over an affine henselian pair follows with using of \cite{Gab94}, or by the direct argument obtained in \cite{PoppiPBchAffHensPair} by Poppi.

The result of \cite{BiKr18} by F.~Binda, A.~Krishna
gives an example closer to the context motivic cohomologies that looks like the rigidity of such type.
Namely, the result relates classical Chow group or relative 0-cycles and Levine-Wiebel Chow groups with finite coefficients on the closed fibre of a projective flat regular scheme over a spectrum of (excellent) henselian discrete valuation ring. 

\subsection{Categorical property.}\label{subsect:SHCatPropertyRigidity} 
The author would appreciate discussions on the question: Could we formulate such a property of the category $\DM(k,\bbZ/l\bbZ)$ with finite coefficients and the $l$-completed stable motivic homotopy category $\SH(k)^\wedge_l$, $(l,\chark k)=1$, 
that would imply Theorem \ref{th:introduction:Rigidity_E:SH->Ab_phi}.

The question can be motivated in view of rigidity theorems for the \'etale versions of motivic categories
$\DM^-_\text{\'et}(k, \mathbb Z/l\mathbb Z)$, $\mathbf{D}^-_{\A^1,\text{\'et}}(k, \mathbb Z/l\mathbb Z)$,
and the $l$-completed category $\SH_\text{\'et}(B)^\wedge_l$ over a base scheme $B$
proven for different cases
by Voevodsky in \cite{VSF00}, 
Ayoub \cite{Ayo13}, 
Deglise-Cesinski \cite{CD16},
and the most recent and the strongest results of such type is \cite{BachmannRigidSHet} by Bachmann.

The mentioned categorical equivalences imply immediately
rigidity isomorphism \eqref{eq:affineAYXBrigidityisomorphism}
for $l$-torsion additive presheaves $\SH_\text{\'et}(B)\to \Ab^\op$
on essentially smooth strictly henselian local schemes.
The case of smooth affine henselian pairs follows with using the rigidity 
for \'etale cohomologies \cite{Gab94}.



There is a categorical property of another type relating to the rigidity isomorphisms proven in \cite{RO08} by R\"ondigs and \O stvaer. It is not discussed in the article, but it looks that by analogy with \cite{RO08} there is a fully faithful embedding $\SH(Z)^\wedge_l\to\SH(X^h_Z)^\wedge_l$ of categories for a smooth affine henselian pair $Z\to X^h_Z$ over a field $k$.

\subsection{The mixed characteristic case.}
The affine henselian pair $(S,x)$ with $S$ being local henselian scheme 
of a mixed characteristic and $x\in S$ the closed point is not a smooth affine henselian pair over $S$, 
since $x$ is not a smooth $S$-scheme.
Secondly, the discussed in the next subsection key ingredient 
does not work over
non-zero dimensional base $S$. 
So in the present work there is no 'cross characteristic' effect such as in \cite{Gab92,Sus83,Sus84}.

\subsection*{Acknowledgment} 
The author is grateful to A.~Ananyevskiy for many helpful discussions on the problem and related questions, for his help with finding the mistakes in the early ideas of the proof.
Also the author thanks 
I.~Panin and F.~Binda for the consultations on the question on the generality of the Picard rigidity statements and applications of the proper base change theorem.
The author is grateful to T.~Bachmann for the form of the result in terms of the $\GW(k)$ coefficients.


\section{Curves (framed curves), and $\A^1$-homotopies.}

In this section we  explain some geometrical basis and motivic categorical nature of the rigidity isomorphism proven in Theorems \ref{th:introduction:Rigidity_E:SH->Ab_phi}, and \ref{th:introduction:Rigidity:framed}.
The section is included to help to navigate in our constructions since the reading experience shown that it could be more complicated without presented comments.
At the same time, we explain the main novelty of the present work.

\subsection{The problem and novelty.}\label{subsect:Intr:AffinepairRelCurve}

\subsubsection{Relative curves and 'curve'-homotopies.}

Given a pair of morphisms of schemes $f_0, f_1\colon U\to X$. 
We call by a \emph{'curve'-homotopy} or a \emph{'curve'-correspondence} connecting $f_0$ and $f_1$ a diagram
\begin{equation}\label{eq:Ue1e2CgX}\xymatrix{
U\ar@<1ex>[r]^{e_0}\ar@<-1ex>[r]_{e_1} & \relcurve\ar[l]|{\mathrm{pr}}\ar[r]^g& X
}\end{equation}
such that $e_0$ and $e_1$ are right inverse to the morphism $\mathrm{pr}$ of relative dimension one, and $g\circ e_1=f_1$, $g\circ e_0=f_0$.
If $\relcurve=\A^1_U$, then \eqref{eq:Ue1e2CgX} is an $\A^1$-homotopy.
In the context of smooth (compact) manifolds, the notion of smooth 'curve'-homotopies would be equivalent to the usual topological homotopies. In algebraic geometry, there are many non-isomorphic (and even smooth) relative curves. 
Nevertheless to construct an $\A^1$-homotopy we can start with a 'curve'-homotopy, and then construct a (relative) $\A^1$-homotopy inside the relative curve $\relcurve$ over $U$.

Using relative curves is a common feature in many reasonings on algebraic K-theory, algebraic groups, and motivic homotopy theory.
Some examples are Voevosky's proof of the Gersten conjecture and other statements on $\A^1$-invariant presheaves with transfers in \cite{VSF00},
and the (stable) connectivity theorem by Morel \cite{Mor1,Mor12}.
The proof of Grothendieck–Serre's conjecture in \cite{GrothSerre-PSV} by Panin, Stavrova, Vavilov, and \cite{Pan-GSConj} by Panin is also one example, where constructing and using relative curves played an important role.

Let us note that constructions or relative curves developed in some of the above references appeared in earlier works on K-theory and cycle groups by other authors.

\subsubsection{Quillen's trick.}


A basic way to turn an affine scheme $X$ of dimension $d$ over a field $k$ into a relative curve is to consider a generic (surjective) map $p\colon X\to \A^{d-1}_k$ of relative dimension one.
If $f\colon U\to X$ is a morphism, then by the base change we can get the diagram 
\[
\xymatrix{
U\ar@<1ex>[r]^{e} & \relcurve\ar[l]|{\mathrm{pr}}\ar[r]^g& X
}
, \; \mathcal C= U\times_{\A^{d-1}_k} X, g\circ e=f,\]
that defines 'the direction' for moving of $U$ inside $X$.

Quillen's trick allows choosing the map $p$ such that $X$ is smooth over $\A^{d-1}_k$ at a given finite set of points, 
and such that a given closed subscheme (of positive codimension) $Z$ in $X$ is finite over $\A^{d-1}_k$.
Moreover, it is possible to compactify $X$ over $\A^{d-1}_k$ in a certain way. 
Quillen's trick is an often-used technique that has many versions, the original reference is \cite[Lemma 5.12]{Qillenstrick}. 
%
%

A variant of Quillen's trick was used in \cite{VSF00} as mentioned above.
In \cite{Mor1,Mor12} it was used a more advanced construction called Gabber's presentation lemma; 
it provides a map $p\colon X\to \A^{d-1}_k$ 
equipped with another additional data and properties, 
see \cite[Lemma 3.1]{Gab94GerstConjsomecomplevancyc}, and in more precise form \cite[Theorem 3.2.1]{CTHK}.


\subsubsection{Key construction.}\label{subsubsect:KeyConstructCurve}
The mentioned above techniques and examples relates 
to the study of local and semi-local (essentially smooth) schemes.
The main difficulty solved in the present article is the construction of the appropriate 'curve'-homotopy for a smooth affine henselian pair.

Let $U=X^h_Z$ for $X\in \SmAff_k$. 
The key construction in Section \ref{sect:ContractingDeduct} gives
a diagram \eqref{eq:Ue1e2CgX}
such that $g\circ e_1\colon U\to X$ equals the canonical morphism $\can\colon U=X^h_X\to X$, 
and $g\circ e_2$ passes through $Z$,
and such that
\begin{itemize}
\item[(0)] 
$\relcurve\to U$ is of dimension one, and smooth over $e_i(U)$, $i=1,2$,
\item[(1)] there is a closed immersion of $U$-schemes $\relcurve\hookrightarrow \ovrelcurve$, 
such that $\ovrelcurve$ is flat projective of dimension one,
and the closed complement $\mathcal C_\infty$ of $\relcurve$ is finite over $U$, 
\item[(2)] 
there is an ample line bundle $\calO(1)$ on $\ovrelcurve$ 
that is trivial on $\mathcal C_\infty$ and on $\relcurve$,
\end{itemize}
We could say that the curve $\relcurve$ retracts $X^h_Z$ to $Z$ inside $X$,
and this 'retraction' leads to rigidity isomorpihsm \eqref{eq:affineAYXBrigidityisomorphism}.

The direct application of Quillen's trick to $X$ (or $U$) over $k$ could brake the smoothness condition (0). 
The application of Quillen's trick to $U$ over $Z$ could destroy the properties of the compactification (1) and (2).
The presented technique does not improve or generalize Quillen's trick but plays a similar role.

The constriction does not use a map $p\colon X\to \A^{d-1}_k$. 
The schemes $\mathcal  C$ and $\overline {\mathcal C}$ are obtained via an appropriate choice of $(N-1)$ equations in the relative affine and projective spaces $\A^N\times U$, and $\PP^N\times U$.
So our technique could be used in much earlier or even classical works, 
but it appears naturally in the motivic homotopy theory with the language of framed correspondences.
The mentioned equations relate to what we call by the framing of a scheme, see Subsection \ref{subsect:FrCorddiminner} for details.

\subsubsection{The local case}
Before we proceed with other details on the proof of rigidity Theorems \ref{th:introduction:Rigidity_E:SH->Ab_phi} and \ref{th:introduction:Rigidity:framed}
let us briefly recall 
the plan of the argument %
in the local case
used in many works listed in Section \ref{subsect:Gabbersrigidityreview} including 
\cite{SV96} and \cite{AD-Rigid}.
\par 
Any essentially smooth local henselian scheme 
with separable residue field $\kappa =k(x)$ 
equals the henselization $(\A^d_{\kappa})^h_{\{0\}}$ 
at the zero point $\{0\}\in \A^d_{\kappa}$.
So there is a sequence of one-dimensional morphisms
\begin{equation}\label{eq:Ai,i-1normalfsequence}(\A^d_{\kappa})^h_{\{0\}}\to (\A^{d-1}_{\kappa})^h_{\{0\}}\to\dots (\A^1_{\kappa})^h_{\{0\}}\to x.\end{equation}
Isomorphism \eqref{eq:affineAYXBrigidityisomorphism} is proven 
by the sequence of homotopies that consequently contract
each $(\A^{i}_{\kappa})^h_{\{0\}}$ to $(\A^{i-1}_{\kappa})^h_{\{0\}}$.

%

Saying accurately 
for a (small enough) \'etale neighbourhood $X$ of $z=\{0\}$ in $\A^i_{\kappa}$, 
and $U=(\A^{i}_{\kappa})^h_{\{0\}}=X^h_z$,
it is constructed a diagram \eqref{eq:Ue1e2CgX}
such that $g\circ e_0=\can\colon U\to X$ is the canonical morphism, and $g\circ e_1$ is a lift of the composite morphism $\A^i\to \A^{i-1}\times \{0\}\hookrightarrow \A^i$.
Then the equality \[[e_0]= [e_1]\in [U, \mathcal C]/l_\epsilon\] for the classes of morphisms in ${\SH(U)}$ is proven. 
The proof of the last equality differs for different categories and types of transfers,
but the critically important properties of $\mathcal C$ 
are 
\begin{itemize}
\item[(0)]
the relative curve $\mathcal C$ is smooth at the point $e_1(x)=e_2(x)$, 
\item[(1)] 
the existence of 
a closed immersion $\mathcal C\hookrightarrow \overline{\mathcal C}$, 
such that 
$\mathcal C_\infty=\overline{\mathcal C}-\mathcal C$ finite over $U$.
\end{itemize}
In addition, it is important that
\begin{itemize}
\item[(2)] 
$\calO(1)$ is trivial on $\mathcal C_\infty$, and $e_i(U)$, $i=0,1$, where
$\calO(1)$ is an ample line bundle on $\ovrelcurve$,
\end{itemize}
though (2) 
holds automatically
because $\mathcal C_\infty$ and $e_i(U)$ are semi-local affine schemes.

%
In \cite{SV96} and \cite{AD-Rigid} the relative curves $\relcurve$ and $\ovrelcurve$ are constructed by variants of 
Quillen's trick. 
Properties (0) and (1)
define the generic position condition on the parameters.
%
%

The first formal obstruction 
for copying of 
the above argument 
in the smooth affine henselian pair case
is that the normal bundle $N_{Z/X}$ is not trivial in general; 
so (formally) there is not the filtration \eqref{eq:Ai,i-1normalfsequence}.
This is fixed in Subsection \ref{subsect:codimonetrivnortangbunddlesreduct}
by replacing of $X$ and $Z$ by such schemes that vector bundles $T_X$, $T_Z$, and $N_{Z/X}$ are trivial. 
The most essential question as mentioned above
is the construction of 
the relative curves over the affine scheme $U=X^h_Z$.

\subsection{Constructing of $\A^1$-homotopies and framed correspondences.}\label{subsect:FrCorddiminner}

Theorem \ref{th:introduction:Rigidity_E:SH->Ab_phi} 
says roughly speaking that 
the morphism 
\[Z\wedge (\mathrm{H}\mathbb Z/\phi)\to X^h_Z\wedge (\mathrm{H}\mathbb Z/\phi) \]
given by the canonical immersion
is an isomorphism in $\SH(k)$, 
where $\mathrm{H}\mathbb Z\in \SH(k)$ denotes the image of the Eilenberg-MacLane spectrum
along the canonical functor $\SH\to \SH(k)$.

For any smooth affine henselian pair 
there is a morphism $r\colon X^h_Z\to Z$ left inverse to the immersion $i\colon Z\to X^h_Z$.
To prove Theorem \ref{th:introduction:Rigidity_E:SH->Ab_phi} means to prove that
the composition 
\begin{equation}\label{eq:retractcomposition}
X^h_Z\wedge (\mathrm{H}\mathbb Z/\phi)\to Z\wedge (\mathrm{H}\mathbb Z/\phi)\to X^h_Z\wedge (\mathrm{H}\mathbb Z/\phi) \end{equation}
equals the identity.
The scheme $X^h_Z$ is not an object in $\Sm_k$ and $\SH(k)$,
so strictly speaking for any \'etale neighbourhood $\widetilde X$ of $Z$ in $\widetilde X$ put at the right side of \eqref{eq:retractcomposition} we find an \'etale neighbourhood $U$ 
put at the left side such that the composite morphism equals the canonical one.

The required equality 
in $\SH(k)$ is proven via 
(a sequence of) 
$\A^1$-homotopies. 
The instrument that allows us 
to make precise constructions of morphisms (and homotopies) in $\SH(k)$
is the framed correspondences technique.
%
%

\subsubsection{Basic principle: equations and elements in homotopy groups.} 


To explain the basic principle we are going to construct a morphism of motivic spaces $\PP^N_{/\infty}\to T^{\wedge N}$ over a scheme $S$,
where
 $\PP^N_{/\infty} = \PP^N_S/\PP_S^{N-1}$ and $T=\A^1_S/(\A^1_S-0)$
 are two models of the motivc sphere over $S$.
%



Given a set of sections $u=(u_1, \dots u_N)$ on $\PP^N_S$ and a regular map $\overline g$
\[u_i\in\Gamma(\PP^N_S, \calO(d_i)),i=1,\dots N, \quad %
\overline g\colon Z(u)\to \{0,1\},\]
and suppose that the vanishing locus $Z(u)$ does not intersect $\PP^{N-1}_S=Z(t_\infty)$.
Then the regular map 
\[(u_1/t_\infty^{d_1}, \dots,u_N/t_\infty^{d_N}, \overline g)\colon \A^N_S\to \{0,1\}\times\A^N_S\]
induces the morphism of motivic spaces in $\mathbf H_\bullet(S)$
\[\PP^N_{/\infty}\to \A^N_S/(\A^N_S-Z(u))\to (\{0,1\}/\{1\})\wedge T^{\wedge N}\simeq T^{\wedge N}.\]


\subsubsection{Language of framed correspondences}\label{subsubsect:FrCorLanguage}
Let 
$P_\infty\hookrightarrow \overline P$, 
$V_\infty\hookrightarrow \overline V$
closed immersions such that 
(1)
$\overline P$ is a projective $S$-scheme, 
$P_\infty = Z(t_\infty)$ for a section $t_\infty$ of an ample sheaf $\calO(1)$ on $\overline P$,
(2)
$\overline V$ is a quasi-projective $S$-scheme, and
$V=\overline V-V_\infty$ is smooth over $S$.
Some data that generalise $(u,\overline g)$ above and satisfy certain conditions 
can be used to define morphism in $\mathbf H_\bullet(S)$
\begin{equation}\label{eq:pePeqVT}\overline P_{/\infty}\to \overline V_{/\infty}\wedge T^c, c\geq 0,\end{equation} 
where
$P_{/\infty}=\overline P/P_\infty$, $V_{/\infty}=\overline V/V_\infty$. 

\emph{Framed correspondences} (defined in \cite{Voev-FramedCorrM} and \cite{GP14})
are sets of data 
that relate to the case of \eqref{eq:pePeqVT}
for 
$P_{/\infty}=\PP^N_{/\infty}$, 
$c=N$, 
and $V_\infty=\emptyset$.

\emph{Framed schemes} or schemes with \emph{framing}, see Definition \ref{def:Zar_normal_framing}, 
relate to the case of \eqref{eq:pePeqVT}
for $P_{/\infty}=\PP^N_{/\infty}$, 
$\overline{V}=Z(u)$, $Y=Z(u)\cap (\overline P-P_\infty)$, $\overline{g}=\id_{\overline Y}$.
In this case $V$ is a smooth closed subscheme in $\A^N_X$ defined by $c$ algebraic equations.


\emph{$P$-inner framed correspondence from $X$ to $V$}, see Definition \ref{def:innerFrnr},
is the set of data relating to the case of \eqref{eq:pePeqVT} for $c=N$, $V_\infty=\emptyset$, and where $P=\overline P-P_\infty$.



\subsubsection{$\A^1$-homotopy.}\label{subsubsection:A1hommotsuspspaces}
In view of the functor \eqref{eq:Fr+->SH} to get the claimed equality for morphism \eqref{eq:retractcomposition} for $\phi=l_\varepsilon$
we connect the classes of morphisms
\begin{equation}\label{eq:icircrid}
U\xrightarrow{r} Z\xrightarrow{i^\prime}  \widetilde X, \quad U\xrightarrow{\mathrm{can}} \widetilde X,
\end{equation}
by a framed $\A^1$-homotopy in $\ZF_*/l_\varepsilon(k)$, 
i.e. a framed correspondence $\A^1\times U\to \widetilde X$.
Without loss of generality we can assume that the \'etale neighbourhood $\widetilde X$ of $Z$ in $X$ is affine of dimension $d$ over $k$.

We start with \emph{a framing} and a compactification $\overline X$ for the affine scheme $\widetilde X$ over $k$. 
Next in Section \ref{sect:ContractingDeduct} we connect morphisms \eqref{eq:icircrid} by a \emph{framed} curve $\relcurve$ over $U$ equipped with a 'fine' compactification $\ovrelcurve$. 
In Section \ref{sect:SectRelCurve} we construct a framed $\A^1$-homotopy inside $\relcurve$.

%
%
This leads to the following sequence of morphisms in the category $\mathbf{H}_\bullet(k)$ (or morphisms of motivic spaces)
that connects the suspensions of morphisms \eqref{eq:icircrid}
\begin{equation}\label{eq:FramingdFrinnerFr}\begin{array}{clll}
\A^1_+\wedge \underbrace{U_+\wedge \PP^{N}_{/\infty}}_{||} \xrightarrow{}& 
\A^1_+\wedge \underbrace{U_+\wedge \widetilde X_{/\infty}}_{||}\wedge T^{N-d} \xrightarrow{} 
&\underbrace{\A^1_+\wedge \mathcal C_{/\infty}}_{||}\wedge T^{N-1}&\xrightarrow{} \widetilde X_+\wedge T^{N}\\
(U\times\PP^{N}/U\times\PP^{N-1}) & (U\times\wtovX/U\times \widetilde X_\infty) &(\A^1\times \overline\relcurve/ \A^1\times \relcurve_\infty)&
\end{array}\end{equation}
Here $\relcurve = \ovrelcurve - \relcurve_\infty$, $\widetilde X= \overline X-\widetilde X_\infty$.

\subsubsection{Set of equations.}
The main essential reason 
of the splitting for the construction in the steps corresponding to morphisms \eqref{eq:FramingdFrinnerFr}
is that it helps to control the process of the special and the generic choices of parameters,
and framed correspondences help to organize the data.
The scheme $\ovrelcurve$ is defined precisely by $(N-1)$ equations in $\PP^N_U$. 
Good properties of the compactification $\ovrelcurve$ is the critical point that allows separating the difficulties in the mentioned steps.

\subsubsection{Trivial tangent bundle}
Any smooth framed scheme $X$ has stably trivial tangent bundle.
That is way to make our construction, we firstly replace $Z$ and the \'etale neighbourhood  $\widetilde X$ in \eqref{eq:retractcomposition} by a motivically equivalent schemes with trivial tangent bundles.

\subsubsection{Smoothness.}
Indistinct to what is assumed 
in \eqref{eq:FramingdFrinnerFr}
the curve $\relcurve$ is actually non-smooth over $U$ in the proof.
So sequence 
\eqref{eq:FramingdFrinnerFr} would be formally contained 
in the motivic homotopy category $\mathbf H_\bullet(\Sch_k)$ (instead of $\mathbf H_\bullet(k)$).
The language of framed correspondences allows us 
to work with 
non-smooth schemes $\relcurve$ over $U$.
At the same time, we need that the morphism $\relcurve\to U$ is smooth over 
closed subschemes in $\relcurve$ that are graphs of morphisms \eqref{eq:icircrid}, 
and it is important that $X$ and $Z$
are smooth.

\subsection*{Notation and conventions}
For a base scheme $S$, denote by $\Sch_S$ the category of schemes over $S$ and by $\Sm_S$ the subcategory of smooth ones. 
For $X\in \Sch_S$ we say that $X$ is of a pure dimension $d$ 
over $S$ and write $\dim X=d$ iff for any point $x\in S$ each irreducible component of $X\times_S x$ has the Krull dimension $d$. 
In the case of $S=\Spec k$ 
or $X\in \Sm_S$ we call the schemes of a pure dimension also by schemes of a constant dimension.

For $X\in\Sm_S$ denote by $T_{X/S}$ and $ T^\vee_{X/S}$ the relative tangent and cotangent bundles of $X$ over $S$. 
Denote by $N_{X/Y}$ and $ N^\vee_{X/Y}$ the normal and conormal bundles for a closed immersion $X\hookrightarrow Y$, $Y\in \Sm_S$. Denote by $\Omega_{X}$ and $\Omega_{X/Y}$ the coherent sheaves of sections of $T_X^\vee$ and $N^\vee_{X/Y}$.

For a scheme $X$ we denote by $X_{red}$ the (maximal) reduced closed subscheme.
For a regular function $f\in \calO_X(X)$ on a scheme $X$ 
we denote by $Z(f)$ the vanishing locus of $f$ and denote by $Z(f)_{red}$ the reduced subscheme of $Z(f)$. 
The same notation we use for the common vanishing locus of a set of functions, and for vanishing loci of sections of vector bundles.
For a vector bundle 
$\mathcal L$ on a scheme $X$ we denote $\mathcal L^\times = \mathcal L-0_X$, where $0_X\subset\mathcal L$ is the zero section. 
We denote by $\mathbf 1_X$ the trivial line bundle on $X$.

For an open immersion $j\colon V\hookrightarrow X$ 
and a regular function $f$ on $X$ 
sometimes we denote the inverse image $j^*(f)$ by the same symbol $f$, 
and similarly for sections of a vector bundle.
Moreover, sometimes we do the same also for \'etale morphisms $j\colon V\to X$ if it is clear from the context.
For a morphism of schemes $f\colon X\to Y$ we write $f^*$ for the base change or the inverse image of any objects (or morphisms) over $Y$ and $X$. 

For closed subschemes $Z_1,Z_2\subset X$ we denote 
$Z_1-Z_2\stackrel{def}{=}Z_1-(Z_1\cap Z_2)$. 
%

\section{Henselian pairs.}\label{subsect:HensPairs}

In the section,
we summarise facts about (affine) henselian pairs. 
In the next section we need Definition \ref{def:HensPair}, Definition \ref{def:SmHensPair}, 
the equivalence of points (2) and (0) in Lemma \ref{lm:AffHenselianPairs}, 
Lemma \ref{lm:smoothretraction}, and Proposition \ref{prop:lrootLineBun}. 

\subsection{Affine henselian pairs}
Let us define henselian pairs as 
the class of closed immersions in the category of schemes
that have left lifting property with respect to the class of \'etale morphisms.
Consider the diagram 
\begin{equation}\label{eq:henseletsmdiaglift}\xymatrix{
Z\ar[d]^i\ar[r] & X^\prime\ar[d]^p\\
U\ar@{-->}[ru] \ar[r]& X
}\end{equation}
\begin{definition}\label{def:HensPair}
A closed immersion of schemes $i\colon Z\to U$ is called \emph{henselian pair} if and only if
for any \'etale morphism $p\colon X^\prime \to X$, 
for any commutative square \eqref{eq:henseletsmdiaglift}
there is a diagonal morphism such that the diagram commutates.
\end{definition}

\begin{example}
In particular, if the base change functor along a closed immersion $i\colon Z\hookrightarrow  U$ 
induces the equivalence of the categories of \'etale schemes over $U$ and $Z$,
then $i$ defines henselian pair.
This is the case if $U$ is affine and the vanishing ideal $I(Z)$ is nilpotent,
by \cite[Lemma 15.11.2]{StacksProject}.

Then by \cite[Lemma 15.11.4]{StacksProject} 
any $I$-adically complete ring $A$ defines the henselian pair $\Spec A/I\hookrightarrow \Spec A$.
\end{example}

A henselian pair $Z\hookrightarrow U$ is called \emph{affine} if both schemes $U$ and $Z$ are affine.
The henselian pairs are usually considered in literature in the affine case, so the term henselian pair usually means affine henselian pair.
The affine henselian pairs have the following list of equivalent definitions.

\begin{lemma}\label{lm:AffHenselianPairs}
Let $A$ be a ring, $I\subset A$ be an ideal; $U=\Spec A$, $Z=Z(I)=\Spec A/I$ be the vanishing locus.
Then 
the following conditions are equivalent:
\begin{itemize}
\item[(0)]
For any quasi-affine smooth morphism $p\colon X^\prime \to X$,
for any commutative square \eqref{eq:henseletsmdiaglift}
there is a diagonal morphism such that the diagram commutes.

\item[(1)]
For any \'etale morphism $p\colon X^\prime \to X$, 
for any commutative square \eqref{eq:henseletsmdiaglift}
there is a unique diagonal morphism such that the diagram commutes.

\item[(2)]
$i\colon Z\hookrightarrow U$ defines a henselian pair.

\item[(3)]
For any integral $A$-algebra $C$ there is an isomorphism $\mathrm{Idem}(C)\simeq \mathrm{Idem}(C\otimes A/I)$, where $\mathrm{Idem}$ denotes the set of idempotents. 

\item[(4)]
For any finite $A$-algebra $C$ there is an isomorphism $\mathrm{Idem}(C)\simeq \mathrm{Idem}(C\otimes A/I)$. 

\item[(5)]
$I$ is in Jacobson redical of $A$,
and for any monic polynomial $f\in A/I[t]$, 
any factorisation $\overline f=\overline g\overline h\in A/I[t]$ with monic generating the unit ideal in $A/I[t]$, where $\overline f$ is the image of $f$, 
admits a lift $f=gh\in A[t]$ with $g,h$ monic polynomials.

\item[(6)] (Gabber)
$I$ is in Jacobson redical of $A$, and any monic polynomial
\begin{equation}\label{eq:tildeTn(1-T)}
f=T^n(1-T)+a_n T^{n-1}+\dots +a_0\in A[t] 
\end{equation} with $a_i\in I$, $n\geq 1$, has a root $\alpha\in 1+I$.

%
\end{itemize}
\end{lemma}

\begin{proof}
The equivalences (3)$\Leftrightarrow$(4)$\Leftrightarrow$(5)$\Leftrightarrow$(6) are given by \cite[Lemma 15.11.6]{StacksProject}.

In each one of the parts (0,1,2) it could be assumed without loss of generality that the morphism $U\to X$ in diagram \eqref{eq:henseletsmdiaglift} is identity.
Actually, the category of lifts in a given diagram \eqref{eq:henseletsmdiaglift} is equivalent to 
the category of lifts in the diagram like \eqref{eq:henseletsmdiaglift} with $X$ replaced by $U$, and $X^\prime$ replaced by $X^\prime\times_X U\to U$.


Assume (4), we are going to prove (2).
Note that for any scheme $Y$ the set $\mathrm{Idem}(\calO_Y(Y))$ equals to the set of splittings $Y=Y_0\amalg Y_1$.
Given a diagram \eqref{eq:henseletsmdiaglift} such that $U=X$. By Zariski's main theorem \cite[Theorem~8.12.6]{GD67} it follows that the morphism $p$ passes throw the composition 
\[X^\prime\xrightarrow{j} \overline{X^\prime}\xrightarrow{\overline{p}} X=U\] 
for an open immersion $j$ and a finite morphism $\overline{p}$. 
Then by point (4) it follows that there is a splitting $\overline{X}^\prime=X_r\amalg \hat X_r$ such that 
\[X_r\times_U Z= r(Z), \hat X_r\times_U Z= \overline{X}^\prime\times_U Z - r(Z).\] So the morphism $p_r\colon X_r\to U$ is finite and surjective, and such that $X_r\times_U Z\cong Z$. Hence $p_r$ is a closed immerison.
Since $p$ is \'etale, it follows that $p_r$ is \'etale over $Z$. Hence $p_r$ is an isomorphism, and $X_r\cong U$. 
The inverse morphism $U\to X_r$ defines the required lift $U\to X$ for $p$.

Assume (2), we are going to prove (6). 
Consider the morphism $p\colon Z(f)\to U$ for the polynomial $f$ given by \eqref{eq:tildeTn(1-T)}.
Then $p$ is \'etale over the closed subscheme $\{1\}\times Z$ in sense of the immersions $\{1\}\times Z\hookrightarrow Z(f)\hookrightarrow \A^1\times U$.
Hence by (2) there is a lift $\alpha\colon U\to Z(f)$ of the morphism $Z\cong \{1\}\times Z\to Z(f)$. The composite regular map \[U\xrightarrow{\alpha} Z(f)\to \A^1\times U\] gives the root $\alpha\in 1+I\subset \calO_U(U)$.


Thus we have the equivalences (2)$\Leftrightarrow$(3)$\Leftrightarrow$(4)$\Leftrightarrow$(5)$\Leftrightarrow$(6).
The implications (0)$\Rightarrow$(1)$\Rightarrow$(2) are immediate.

To prove the equivalence of (1) and (2) 
we are going to show that if for all squares \eqref{eq:henseletsmdiaglift} in (2) there exists a lift, then for each such square the lift is unique. 
It follows from (2) that $Z$ has nonempty intersection with each irreducible component of $U$.
Indeed, if there is a Zariski open neighbourhood $U^\prime$ of $Z$ in $U$ then by the lifting property it follows that the canonical immersion $U^\prime\to U$ has a left inverse, and hence $U=U^\prime$.
Assume now that we are given with square \eqref{eq:henseletsmdiaglift} such that $U=X$, and there are two lifts $s_0,s_1\colon U\to X^\prime$. Then since $s_0$ and $s_1$ are left inverse to the canonical projection $X^\prime\to X=U$, it follows that the images $s_1(U)$ and $s_0(U)$ are closed subschemes. 
Moreover, since $p$ is \'etale, $s_1(U)$ and $s_0(U)$ are clopen subschemes in $X^\prime$. Finally, since $s_0(U)\cap s_1(U)\supset i^\prime(Z)$, and $Z$ intersects with each irreducible component of $U$, it follows that $s_1(U)=s_0(U)$.

The implication from (2) to (0) follows by Lemma \ref{lm:QAffSmoothLift} proven in Subsection \ref{subsection:LiftingPropertyHensePairs}.
%
%
\end{proof}

\subsection{Lifting property}\label{subsection:LiftingPropertyHensePairs}

By Definition \ref{def:HensPair} henslian pairs satifies the left lifting property with respect to \'etale morphisms.
The class of smooth morphisms in the category of affine schemes of finite type over an affine 
base scheme $S$
can be defined via the right lifting property with respect to the class of affine henselian pairs.



\begin{lemma}\label{lm:etaledifferential}
Given a 
smooth $S$-scheme $X\in \Sm_S$ of relative dimension $d$ at a point $x\in X$ over 
a noetherian base scheme $S$.
Let $f=(f_1,\dots, f_d)$ be a vector of regular functions on $X$ such that 
the differential $\mathrm{d}_x f$ of $f$ at $x$ defines an isomorphism $T_x X\cong \mathbf{1}^d_x$.
Then $Z(f)$ is \'etale over $S$ at $x$.
\end{lemma}
\begin{proof}
We deduce the claim from \cite[Lemma 29.35.15.(8)]{StacksProject}.
Without loss of generality we can assume that $S$ is affine, and $X$ is smooth affine over $S$.
Moreover, shrinking $X$ to a Zariski open neighbourhood of $x$ we get that the tangent bundle $T_X$ is trivial.

Then there is a closed immersion $X\hookrightarrow \A^N_S$. Since $T_X$ is trivial, the normal bundle $N_{X/\A^N_S}$ is stably trivial. 
Hence for some integer $N^\prime>N$ we get a closed immersion $X\hookrightarrow \A^{N^\prime}_S$ that the normal bundle is trivial. 
Indeed, if $N_{X/\A^N_S}\oplus \mathbf{1}^{N^\prime - N}\simeq \mathbf{1}^{N-d}$, then $X\hookrightarrow X\to \A^N_S=\A^N_S\times_S \{0\}_S\to \A^{N^\prime}_S$ is the required immersion.

Choose a regular functions $f_{d+1}, \dots, f_N$ on $\A^{N}_S$ such that the differential of the vector $(f_{d+1}, \dots, f_{N})$ induces the isomorphism $N_{X/\A^N_S}\cong \mathbf{1}^{N-d}_X$.
Thus the local scheme $Z_x$ of the vanishing locus $Z=Z(f_1,\dots f_d)\subset X$ at the point $x$ equals the vanishing locus $Z(f_1,\dots,f_N)$ in $\A^N_S$.
Now by part (8) of \cite[Lemma 29.35.15]{StacksProject} it follows that $Z(f_1,\dots,f_N)$ is \'etale over $S$.
\end{proof}

\begin{lemma}\label{lm:TTNsplitting}
Given a closed immersion of smooth affine schemes $i\colon X\hookrightarrow Y$ over an affine base scheme.
Then there are a splittings $i^*(T_{Y})\simeq T_X\oplus N_{X/Y}$,
where $T_X$ and $T_{Y}$ denotes the tangent bundles and $N_{X/Y}$ denotes the normal bundle.
\end{lemma}
\begin{proof}
Since $X$ and $Y$ are smooth there is a short exact sequence of vector bundles on $X$:
$T_X\hookrightarrow i^*(T_{Y})\twoheadrightarrow N_{X/Y}$.
The last sequence splits, since $X$ is affine,
and the category of finite rank vector bundles on $X$ is equivalent to the category of finite rank projective modules over the ring of regular functions $\calO_X(X)$.
\end{proof}

\begin{lemma}\label{lm:SmSMtrivtangbundle}
For any smooth affine scheme $X$ over an affine scheme $S$ there is a smooth affine scheme $X^\prime$ with trivial tangent bundle and morphisms $X\to X^\prime\to X$ that composite morphism equals identity.
\end{lemma}
\begin{proof}
Without loss of generality we can assume that $S$, and $X$ are connected.
Then $X$ is the pure dimension $\dim_S X=d\in \mathbb Z$ over $S$.

By Lemma \ref{lm:TTNsplitting} we have 
$ T^\vee_X\oplus  N^\vee_{X/\A^N_k}\simeq i^*( T^\vee_{\A^N_k})=\mathbf 1_X^n$. 
Denote $\widetilde X= N_{X/\A^N_k}^\vee$; let $p\colon \widetilde X\to X$, and $z\colon X\to\widetilde X$ be the projection, and by the zero section.
Let $\widetilde r\colon Z\to \widetilde X$ be the composition $Z\xrightarrow{r}X\xrightarrow{z}\tilde X$.
Then $\widetilde X$ has a trivial tangent bundle of the rank $n$; indeed, $T_{\widetilde X}^\vee\simeq p^*( T_X^\vee )\oplus T_{X/\A^N_k}^\vee\simeq \mathbf 1^n_{\widetilde X}$. 
So by the above  there is a lift $\widetilde e\colon U\to \widetilde X$ of $\widetilde r$. 
Then $e=p\circ \widetilde e$ is the required lift of $r$.
\end{proof}

\begin{lemma}\label{eq:ConnComponentsHensepair}
Given a henselian pair $Z\hookrightarrow U$. Then if $Z$ is connected, then $U$ is connected.
\end{lemma}
\begin{proof}
In the affine case the claim follows immediate from Lemma \ref{lm:AffHenselianPairs}.(4) applied to $C=A=\calO_U(U)$.
In general, we deduce the claim from Definition \ref{def:HensPair}.
Let $Z=Z_1\amalg Z_2$. Consider the Zariski open neighbourhood $U-Z_2$ of $Z_1$ in $U$.
Then by the lifting property in sense of square \eqref{eq:henseletsmdiaglift} for $X^\prime=U-Z_2$, $X=U$, we see that there is an immersion $U\to U-Z_2$. Hence $Z-2=\emptyset$.
\end{proof}

\begin{lemma}\label{lm:AffSmoothLift}
Let $Z\hookrightarrow U$ be an affine henselian pair.
Let $X$ be a smooth affine $U$-scheme.
Then for any morphism of $U$-schemes $r\colon Z\to X$ there is a lift $e\colon U\to X$.
\end{lemma}
\begin{proof}[Proof of Lemma \ref{lm:AffSmoothLift}]

Without loss of generality by Lemma \ref{eq:ConnComponentsHensepair} we can assume that $Z$, $U$, and $X$ are connected. 
By Lemma \ref{lm:SmSMtrivtangbundle} any smooth affine morphism $X\to U$ is a retract of a smooth affine morphism $X^\prime\to U$ with trivial tangent bundle. So without loss of generality we can assume that the tangent bundle of $X$ is trivial of a constant rank $d=\dim_U X$.

Consider a closed subscheme $\Gamma\subset X$ that is the graph of the morphism $r$ over $U$.
Since $T_X$ is trivial, there are functions 
$f=(f_1,\dots f_d)\in \calO(X\times_U Z)^d$, 
such that $f\big|_Z=0$ and the differential $df\colon N_{\Gamma/X\times_U Z}^\vee\to \mathbf{1}^d_\Gamma$ is an isomorphism. 
Denote $X^\prime=X-(Z(f)-\Gamma))$.
Since the scheme $X$ is affine,
there is 
$\widetilde f\in \calO(X\times_S U)^d$ that is a lift of $f$.

%
By Lemma \ref{lm:etaledifferential} $Z(\widetilde f)$ is \'etale over $U$; 
Since $Z\to U$ is a henselian pair, 
there is a lift $e^\prime\colon U\to Z(\widetilde f)$ 
in the diagram
\[\xymatrix{ Z\ar[dr] & \Gamma\ar[l]^{\simeq}\ar[d]\ar@{^(->}[r] & Z(\widetilde f)\ar[r]^j\ar[d] & X^\prime\ar[dl]\ar[r]^v & X\ar[dll]
\\ 
& U\ar@{=}[r]\ar@{-->}[ur]^{e^\prime} & U
.}\]
The composite morphism $e=v j e^\prime$ is the required lift.
\end{proof}

\begin{lemma}\label{lm:QAffSmoothLift}
Let $Z\hookrightarrow U$ be an affine henselian pair.
Let $X$ be a smooth quasi-projective $U$-scheme.
Then for any morphism of $U$-schemes $r\colon Z\to X$ there is a lift $e\colon U\to X$.
\end{lemma}
\begin{proof}
Let $X=X_\mathrm{aff}-Y$ for an affine $U$-scheme $X_\mathrm{aff}$, and a closed subscheme $Y$.
Let $f\in \calO_{X_\mathrm{aff}}(X_\mathrm{aff})$ be such that $f\big|_Y=0$, $f\big|_{r(Z)}=1$.
Set $X^\prime=X_\mathrm{aff}-Z(f)$. Then $X^\prime$ is a Zariski open neighbourhood of $r(Z)$ in $X$ that is affine over $U$.
Now the cliam follows by Lemma \ref{lm:AffSmoothLift}
\end{proof}

\subsection{\'Etale neighbourhoods and henselizations}

\begin{definition}
Let $Z\subset X$ be a closed immersion of schemes.
An \emph{etale neighbourhood} of $Z$ in $X$ 
is a closed immersion $Z\subset X^\prime$, 
and an etale morphism $X^\prime\to X$ 
such that the diagram 
\[\xymatrix{
& X^\prime\ar[d]\\
Z\ar@{^(->}[r]\ar[ur] & X
}\]
commutes.
\end{definition}
\begin{remark}
For any \'etale neighbourhood $X^\prime \to X$ there is an open immersion $X^{\prime\prime}\to X^\prime$ 
such that $Z\simeq X^{\prime\prime}\times_X Z$.
\end{remark}

Given a closed immersion $Z\to X$.
The identity map $X\to X$ is an \'etale neighbourhood. For any two neighbourhoods $X^\prime\to X$, $X^{\prime\prime}\to X$ the morphism $X^{\prime}\times X^{\prime\prime}\to X$ is a neighbourhood.
So there is a filtering set of \'etale neighbourhoods of $Z$ in $X$.

\begin{definition}
Let $Z\hookrightarrow X$ is a closed immersion of affine schemes.
The projective limit 
\begin{equation}\label{eq:XhZlim}
X^h_Z=\varprojlim_{X^\prime\to X} X^\prime,
\end{equation}
where $X^\prime$ runes over the filtering set of \'etale neighbourhoods of $Z$ in $X$
is called \emph{henselization of $X$ along $Z$.}
\end{definition}

It follows immediate from the definitions that $Z\hookrightarrow X^h_Z$ is a henselian pair.
The limit in \eqref{eq:XhZlim} exists in the category of schemes for any 
$X$ and $Z$.
Since the result of the article deals with the affine schemes case, we write the proof for this situation.

\begin{lemma}\label{lm:XhZAffexists}
For any affine schemes $X=\Spec A$ and $Z=\Spec A/I$ 
the projective limit \eqref{eq:XhZlim} exists, and the scheme $X^h_Z$ is affine, $X^h_Z=\Spec A^h_I$. 
\end{lemma}
\begin{proof}
Any closed subscheme in an affine scheme has an affine Zariski open neighbourhood.
Hence any \'etale neighbourhood $X^\prime\to X$ of the closed subschemes $Z$ in the affine schemes $X$ has a shrinking given by an affine scheme. 
The first claim follows, 
since for any filtering system of affine schemes $X_i=\Spec A_i$, $I\in I$, the projective limit $\varprojlim_{i\in I} X_i$ equals $\Spec \varinjlim_{i\in I} A_i$. 
%
\end{proof}

\begin{remark} 
$A^h_I$ is 
the subring in the $I$-adical completion $\hat A_I=\varprojlim_n A/I^n$
such that $\alpha\in A^h_I$ if $f(\alpha)=0$ for some polynomial $f(T)\in A[t]$ such that $\overline{\alpha}\in A/I$ is a separable root of $\overline f(T)\in A/I[T]$.
\end{remark}

\subsection{Smooth affine henselian pairs}

A \emph{smooth affine henseian pair} $Z\hookrightarrow U$ over a scheme $S$ is an affine henselian pair such that $Z$ is smooth, 
and $U$ is essentially smooth.
Equivalently, smooth (affine) henselian pairs is the henselization of smooth (affine) closed pair.

\begin{definition}\label{def:SmHensPair}
\emph{A smooth (affine) henselian pair} over a base scheme $S$ is a henselian pair of the form $Z\to X^h_Z$ for a closed immersion of smooth (affine) schemes $Z\hookrightarrow X$ over $S$. 
\end{definition}


In the case of smooth affine henselian pair $Z\to U$ over a base scheme $S$
there is a natural (functorial) lift in diagrams \eqref{eq:henseletsmdiaglift}
for all smooth quasi-affine morphisms $X^\prime\to X$ (simultaneously).
%
%
The existence is already proven in subsection \eqref{subsection:LiftingPropertyHensePairs}, and the functoriality follows 
since there is the universal square 
given by $X^\prime=Z$, $X=S$.

In other words there is a retraction 
$r\colon U\to Z$, $r\circ i=\id_Z\colon Z\to Z$. 
Moreover, such a morphism $r$ is essentially smooth.
So, in the case of smooth affine henselian pairs 
the result of the previous section can be formulated equivalently as follows.

\begin{lemma}\label{lm:smoothretraction}
Given a closed immersion $i\colon Z\to X$ in $\Sm_S$ over a base $S$, such that both schemes $X$ and $Z$ are affine; let $U= X^h_Z$.
%
Then there is an \'etale neighbourhood $\widetilde X$ of $Z$ in $X$,
and a smooth morphism $r\colon \widetilde X\to Z$
such that 
\[r\circ i=\id_Z,\]
where $i\colon Z\hookrightarrow \widetilde X$ is the closed immersion.

Moreover, if the tangent bundle $T_Z$ and the normal bundle $N_{Z/U}$ are trivial, 
then there is $\widetilde X$ as above and such that
the relative tangent bundle $T_{\widetilde X/Z}$ is trivial.
\end{lemma} 
\begin{proof}
Since $Z\hookrightarrow U$ is an affine henselian pair, and $Z$ is smooth over $S$ 
by assumption, 
Lemma \ref{lm:AffHenselianPairs}.(5) applied to 
the square
\[\xymatrix{ Z\ar[r]^{\id_Z}\ar[d]^i & Z\ar[d]\\ X\ar[r] & S},\]
gives a retraction $r^\prime\colon U\to Z$, $r^\prime\circ i=\id_Z\colon Z\to Z$. 
Then, since $U=X^h_Z$, it follows that the map $r^\prime\colon U\to Z$ is induced by the map $r\colon \widetilde X\to Z$ for some (affine) \'etale neighbourhood $\widetilde X$ of $Z$ in $X$.

The differential of the map $r$ on the subscheme $Z$ induces the surjective homomorphism on tangent bundles $i^*(T_X)\to T_Z$. 
Hence $r$ is smooth over the subscheme $Z$ in $U$. 
Whence shrinking $\widetilde X$ to a Zariski open neighbourhood of $Z$ in $\widetilde X$
we get that $r$ is smooth.

We are proceed with the rest additional assertion in the lemma. Let $T_{\widetilde X/Z}$ be the relative tangent bundle of the morphism $r$.
Then \[T_{\widetilde X/Z}\big|_Z\cong N_{Z/\widetilde X}\cong N_{Z/U}\]
is trivial.
Since $\widetilde X$ is affine, it follows that
$T_{\widetilde X/Z}$ is trivial on some Zariski open neighbourhood of $Z$. 
So shrinking $\widetilde X$ again
we get that $T_{\widetilde X/Z}$ is trivial.
\end{proof}



\subsection{Division in the Picard group}

Consider an affine henselina pair $i\colon Z\to U$, and a projective $U$-scheme $P$.
Then the closed immersion $i_P\colon P\times_U Z\to P$ is a henselian pair.
Formally, we do not use this fact, but use some rigidity results on this henselian pair,
that follows from 
the proper base change theorem for \'etale cohomologies with coefficients a torsion sheaf,
and its affine analogue proven in \cite{Gab94,HEtfin,Elkiksoleqhens}. 
\begin{lemma}\label{lm:GmrepressetRig}
Let $V$ be a henselian local scheme, and $z\in V$ be the closed point.
Then $\mu_l(z)\cong \mu_l(V)$ for all inverers $l$ inverible at $z$.
\end{lemma}
\begin{proof}
Since $l$ is invertible at $z$, $l$ is invertible on $V$, and on any scheme over $V$. 
Then the morphism of $V$-schemes $l\colon \mathbb{G}_m\to \mathbb{G}_m; t\mapsto t^l$ is \'etale.
The prehseaf $\mu_l$ on $\Sch_V$ is represented by the scheme 
$\mu_l=\mathbf{G}_m\times_{l,\mathbf{G}_m,1}V$, where $1\colon V\to \mathbf{G}_m$ is the closed immersion given by the unit function.
Thus the scheme $\mu_l$ is \'etale over $V$.
Since the sheaf $\mu_l$ on $\Sch_U$ is represented by an \'etale $V$-scheme, the claim follows by Lemma \ref{lm:AffHenselianPairs}.(1).
\end{proof}

\begin{theorem}[Consequence of proper base change theorem and the affine analogue]
\label{th:RigidPUZmul}
Given an affine henselian pair $Z\to U$, and a projective morpism $P\to U$.
Let $F$ be an \'etale sheaf of the category of $U$-schemes, such that for any local strictly henselian scheme, $F(V)\simeq F(z)$, where $z$ is the closed point.

Then for any $l\in \mathbb Z$ ivertible on $Z$, the canonical morpihsm 
$H^i_{\et}(P\times U Z,F)\to H^i_{\et}(P,F)$
is an isomorhism for all $i\in \mathbb Z$.
\end{theorem}
\begin{proof}

Since the higher direct images of a torsion abelian sheaves are torsion, 
combining the proper base change theorem, see 
\cite[Theorem 58.87.11]{StacksProject} or \cite[ch~4, Corollary~2.7]{Milne-EtCoh},
applied to the closed immersion $i\colon Z\to U$, and the affine analogy, see 
\cite[Theorem 58.79.7]{StacksProject}, applied to the same morphism we get
\[ H^i_\et(P\times_U Z, i^*_P(F\big|_{\Et_P}))\simeq H^i_\et(P,F\big|_{\Et_P}), i\in \mathbb Z,\]
where $\mu_l$ denotes the sheaf on $\Et_{P}$, and $i^*_P$ denote the morphism of small \'etale sites.

Let $V$ be (strictly) henselian local scheme over $U$, and $z\in U$ be the closed point.
Then by assumption $F(z)\cong F(V)$.
Since strictly henselian local schemes defines an enough set of points for the \'etale topology, 
the \'etale sheaf $F$ on the small \'etale site $\Et_{P_Z}$ over $P_Z=P\times_U Z$ equals the inverse image $i_P^*(F\big|_{\Et_P})$ of the \'etale sheaf $F\big|_{\Et_P}$ on the small \'etale site $\Et_P$.
So the claim follows.
\end{proof}

In the next sections we use the following consequenct that follows from the above theorem by the same argument as in \cite[sect 2]{SV96} or \cite[lemma 5.1]{AD-Rigid} for local henselian case.
\begin{proposition}\label{prop:lrootLineBun}
Let $i\colon Z\hookrightarrow U$ be an affine henselian pair.
Let $\relcurve\to U$ be a projective flat morphism of pure dimension one.
Let $n\in \mathbb Z$ invertible on $Z$. 

Assume $\calL$ is a line bundle on $\relcurve$ such that $i^*(\calL)\simeq \calO_Z$ 
Then for any $l\in \mathbb Z$, $l\in \calO(U)^\times$ there is $\calL^\prime$ such that $(\calL^\prime)^l\simeq\calL$, and $i^*\calL^\prime$ is trivial.
\end{proposition}
\begin{proof}
The sheaf $\calL$ defines an element in the Picard group on $\mathcal C$.
$\Pic(-)\simeq H^1_\Zar(-,\Gm)\simeq H^1_{et}(-,\Gm)$, 
where the second equality is the Hibert theorem 90 \cite[ch 3]{Milne-EtCoh}. 
The short exact sequence $0\to\mu_l\to\Gm\xrightarrow{l}\Gm\to 0$ of etale sheaves of groups on $\Sch_U$, 
 and the inverse image along the immersion $C\to \relcurve$, where $C=\relcurve\times_U Z$, 
  give the the morphism of long exact sequences
$$\xymatrix{
\ar[r]&H^1_{et}(\ovrelcurve,\mu_l)\ar[r]\ar@{=}[d]&\Pic(\ovrelcurve)\ar[r]^{\tilde p_l}\ar[d] & \Pic(\ovrelcurve)\ar[r]\ar[d]& 
H^2_{et}(\ovrelcurve,\mu_l)\ar@{=}[d]\ar[r]& \\
\ar[r]&H^1_{et}(C,\mu_l)\ar[r]&\Pic(C)\ar[r]^{p_l} & \Pic(C)\ar[r]& 
H^2_{et}(C,\mu_l)\ar[r]&
}$$
where the left and the right vertical isomorphisms are provided by Theorem \ref{th:RigidPUZmul} and Lemma \ref{lm:GmrepressetRig}.
Hence $$\Ker(\tilde p_l)\twoheadrightarrow \Ker(p_l), \Coker(\tilde p_l)\hookrightarrow \Coker(p_l)$$
The inclusion at the right side above implies the claim.
\end{proof}

\section{Framed presheaves.}\label{subsect:FrCor}




In the article we construct some morphisms and $\A^1$-homotopies  
in the stable motivic homotopy category $\SH(S)$ appeared in works Morel and Voevodsky \cite{MV99,Mor0} and Jardine \cite{Jar00}.
We use so called framed correspondences as a tool that gives a precise construction for morphism of a certain type in $\SH(S)$.

The framed correspondences ware introduced in \cite{Voev-FramedCorrM} by Voevodsky, 
and then studied and used as the basic ingredient by Garkusha and Panin in \cite{GP14}.
The pointed set of framed correspondences of the level $n$ from a 
scheme $X$ to a scheme $Y$ 
equals the (pointed) hom-set in the category of pointed Nisnevich sheaves
\begin{equation}\label{eq:Voevodsky'sLemma}\Fr_n(X,Y)\simeq \mathrm{Shv}_\bullet(X_+\wedge (\PP^1/\infty)^{\wedge n}, Y_+\wedge (\A^1/(\A^1-0))^{\wedge n}),\end{equation} 
The basic computational power of framed correspondences 
is provided by the precise description of the set $\Fr_n(X,Y)$ defined by \eqref{eq:Voevodsky'sLemma}.
Def. \ref{def:Fr} comes from the mentioned description,
while the isomorphism \eqref{eq:Voevodsky'sLemma} is called Voevodsky's lemma on framed corr. 
In what follows we don't use formally isomorphism \eqref{eq:Voevodsky'sLemma}, 
but only Def. \ref{def:Fr}, and morphism in \eqref{eq:Voevodsky'sLemma} from the left side to the right that is recalled below the definition. 

Let us note that framed correspondences ware basically considered in the base field case, but the 
constructions we use are defined word by word
for any noetherian and separated base scheme $S$.

\begin{definition}\label{def:Fr}
Let $X,Y$ be smooth schemes over the base scheme $S$.
An \emph{explicit framed correspondence} of a level $n$ over $S$ is a set of data $(Z,\mathcal V, \varphi, g)$ where 
\begin{enumerate}
\item $Z \subset \A^n$ is a closed subscheme, 
\item $e\colon \mathcal V\to \A^n_X$ is an \'etale morphism such that $e^{-1}(Z)\simeq Z$,
\item $\varphi=(\varphi_i)$, $0<i\leq n$, $\varphi_i$ are regular function on $\mathcal V$ such that $Z(\varphi_1,\dots \varphi_n)=Z$,
\item and $g\colon \mathcal V\to Y$ is a morphism of $S$-schemes.
\end{enumerate}
%
\end{definition}

It follows that functions $\varphi$ above defines a set of functions on the henselization $(\A^n_X)^h_Z$.
The pointed set of framed correspondences $\Fr_{n}^S(X,Y)$ for $X,Y\in \Sm_S$ 
is the set of classes of explicit framed correspondences up to the equivalence relation with respect to the choice of the \'etale neighbourhood $\mathcal V$ of $Z$.

\begin{definition}
Define the composition  of framed correspondences
$\Fr_n(X,Y)\times \Fr(Y,V)\to \Fr_{n+l}(X,V)$
\[((Z_1,\mathcal V_1, \varphi, g),(Z_1,\mathcal V_1, \varphi, g))\mapsto (Z_1,\mathcal V_1\times_Y \mathcal V_2, \varphi^\prime_1,\varphi_2^\prime, g_2^\prime)\]
where $\varphi_1^\prime, \varphi_2^\prime$ are inverse images of $\varphi_1$ and $\varphi_2$ with respect to the projections 
$\pr_i\colon \mathcal V_1\times_Y\mathcal V_2\to \mathcal V_i$, $i=1,2$, respectively, and $g^\prime=g_2\circ \pr_2$.  

Denote by $\Fr_+(S)$ the category enriched over (graded) pointed sets
 with objects being a smooth $S$-schemes and morphisms
\[\Fr_+(X,Y) = \bigvee_{n\geq 0}\Fr_n^S(X,Y), X,Y\in \Sm_S.\]
\end{definition}
\begin{remark}
In difference to \cite[def. 2.1]{GP14} in the above definition of explicit framed correspondences
 $Z$ is not a reduced scheme in general, sine we require scheme-theoretic equality $Z(\varphi_1,\dots \varphi_n)=Z$ instead of the set-theoretic $\bigcap_i \{\varphi_i=0\}=Z$.
The set $\Fr_n^S(X,Y)$ is the same in both cases.
\end{remark}


To describe the morphism from the left side to the right in \eqref{eq:Voevodsky'sLemma}
consider an explicit framed correspondences
$(Z, \mathcal V, \varphi, g)$ from $X$ to $Y$;
then the morphism of the pointed sheaves 
$X_+\wedge (\PP^1/\infty)^{\wedge l}\to Y_+\wedge (\A^1,\A^1-0)^{\wedge l}$
is defined by the covering and morphisms 
\begin{gather*}
(X\times (\PP^1)^n)-Z) \cup (X\times \A^n) = X\times (\PP^1)^n\\
\begin{array}{llll}
&(X\times (\PP^1)^n)-Z)&\to& *,\\
(g,\varphi)\colon &(X\times \A^n)&\to& Y\times \A^n.
\end{array}
\end{gather*}

The composite map
\begin{multline}\label{eq:FrSchnisSH[P,T]}\Fr_b(X,Y)\to \mathrm{Shv}_\bullet(X\wedge (\PP^1/\infty)^{\wedge n}, Y\wedge (\A^1/(\A^1-0))^{\wedge n})\to\\ 
[X_+\wedge T^{\wedge n}, Y_+\wedge T^{\wedge n}]_{\SH(S)}=[X_+, Y_+]_{\SH(S)}\end{multline}
defines the functor $\Fr_+(S)\to \SH(S)$.
\par Any regular map $g\colon X\to Y$ defines a framed correspondence of the level $0$, given by $(X,X,,g)$ 
(with the empty set $\{\varphi_i\}$). This gives a functor $\Sm_S\to \Fr_+(S)$.
So the canonical functor $\Sm_S\to\SH(S)$ passes through
\begin{equation}\label{eq:Fr_+->SH}\Sm_S\to \Fr_+(S)\to \SH(S).
\end{equation}

\begin{definition}
An presheaf on the category $\Fr_+(S)$  
is called \emph{a presheaf with framed transfers (or a framed presheaf)}.

Let $c\in \Fr_n(X,Y)$ and $F\colon \Fr_+(S)\to \mathrm{Set}^\mathrm{op}$ be a framed presheaf, denote by $c^*\colon F(Y)\to F(X)$ the induced morphism.
\end{definition}
It follows that any presheaf $E\colon \SH(S)\to \Ab^\op$ defines a presheaf with framed transfers 
\begin{equation}\label{eq:ZFr->Ab}F\colon \Fr_+(S)\to \Ab^\op.\end{equation}

\subsection{$\sigma$-stability and linear framed correspondences.}

Different framed correspondences from $X$ to $Y$ can define the same class in $[X,Y]_{\SH(S)}$.
The following modifications 
of the set $\Fr_+(X,Y)$ 
reflects some structures and equivalences relations in $\SH(S)$. 

Shortly speaking Definitions \ref{def:ZF}-\ref{def:ovZF} give the sequence of functors
\begin{equation}\label{eq:SmFrZFovZSH}\Sm_k\to \Fr_+(S)\to \ZF_*(S)\to \ovZF_*(S)\to \SH(S),\end{equation}
where the category $\ZF_*(S)$ is the additivisation of $\Fr_+(S)$,
and $\ovZF_*(S)$ is the quotient-category that morphisms are $\A^1$-homotopy classes.

The categories $\Fr_+(S)$, $\ZF_*(S)$, and $\ovZF_*(S)$ are graded categories.
At the same time the framed correspondences of different levels 
can define the same class of morphisms $\SH(S)$, namely such an equivalence is given by the $\sigma$-suspension, see the next definition. 

\begin{definition}\label{def:Frsigma}
Denote by $\sigma_X\in \Fr_1(X,X)$ the framed correspondence given by \[(0\times X, \A^1\times X, t, \pr_X),\] where $\pr_X\colon \A^1\times X\to X$ is the projection, $t$ is the coordinate on $\A^1$.
Define presheaves $\Fr(-,Y)$ 
\[\Fr_+(-,Y)\to \Fr(-,Y)=\varinjlim \Fr_n(-,Y),\] 
where the morphisms in the inductive limit are given by the maps 
\[\Fr_n(-,Y)\to \Fr_{n+1}(-,Y)\colon c\mapsto \sigma_Y\circ c.\]
\end{definition}
\begin{definition}\label{def:ZF}
Denote by  $\ZF_n(X,Y)$ the abelian group generated by the classes of level $n$ framed correspondences between $X$ and $Y$ and relations
\begin{equation}\label{eq:additivity_proposrty}
[Z_1,\mathcal V-Z_2,\varphi\big|_{Z_1},g\big|_{Z_1}]+[Z_2,\mathcal V-Z_1,\varphi\big|_{Z_2},g\big|_{Z_2}]=
[Z,\mathcal V-Z,\varphi\big|_{Z},g\big|_{Z}]
\end{equation} 
\end{definition}

\begin{definition}\label{def:ovZF}
Define \[\ovZF_n(X,Y)=\Coker( \ZF_n(X\times\A^1,Y)\xrightarrow{[i_0]-[i_1^*]} \ZF_n(X,Y) ),\]
where $i_0,i_1\colon X\to X\times\A^1$ are the maps given by $X\times 0$ and $X\times 1$.
\end{definition}

Denote by $\ZF_*(S)$ and $\ovZF_*(S)$ the graded additive category with the same objects as $\Sm_S$ and morphisms $\bigoplus_{n\geq 0}\ZF_n(X,Y)$ and  $\bigoplus_{n\geq 0}\ovZF_n(X,Y)$.

Define the groups $\ZF(X,Y)=\varinjlim_{n} \ZF_n(X,Y)$ and $\ovZF(X,Y)=\varinjlim_{n}\ovZF_n(X,Y)$ 
with respect to the left $\sigma$-stabilisation, similarly to Def. \ref{def:Frsigma}.
There is a sequence of morphisms of presheaves on $\Sm_S$
\[
\mathrm{Map}(-,Y)\to \Fr(-,Y)\to \ZF(-,Y)\to \ovZF(-,Y) \to [\Sigma^\infty_{\PP^1}(-),\Sigma^\infty_{\PP^1}Y]_{\SH(k)}.
\]

\begin{definition}
%
%
%


A \emph{linear framed presheaf} over $S$ is 
a functor $\ZF_*(S)\to \Ab^\op$.

A framed presheaf is called $\sigma$-\emph{stable} if for any $X\in \Sm_S$ the framed correspondence $\sigma_\sigma(\id_X)\in \Fr_1(X,X)$ induces the isomorphism $\sigma_X^*\colon F(X)\simeq F(X)$.

\emph{A homotopy invariant linear framed presheaf} is an additive functor $\ovZF_*(S)\to \Ab^\op$. 
\end{definition} 

%

\begin{definition}
Let $\mathcal C$ be a category with finite coproducts. 
A functor $F\colon \mathcal C\to \Ab^\op$ on a category $\mathcal C$ is \emph{additive} if $F(X_1\amalg X_2)\simeq F(X_1)\oplus F(X_2)$.
\end{definition}
\begin{remark}
It follows form the definition that any linear framed presheaf $F$ is \emph{additive}, i.e.
$F(X_1\amalg X_2)=F(X_1)\oplus F(X_2)$ for any $X_1,X_2\in \Sm_S$.
\end{remark}

\begin{lemma}
Assume that $F$ is a homotopy invariant linear $\sigma$-stable framed presheaf.
Let $c_1,c_2\in \ZFr_*(X,Y)$ such that there is an equality of the images $[c_1]=[c_2]$ in $\ovZF(X,Y)$. Then 
the induced morphisms $c_1^*=c_2^*\colon F(Y)\to F(X)$ are equal. 
\end{lemma}
\begin{proof}
The claim follows from definitions.
\end{proof}

\begin{lemma}
\label{lm:EovZFSHandstable}
For any additive functor $E\colon \SH(S)\to \Ab^\op$ there is a homotopy invariant $\sigma$-stable linear framed additive presheaf $F$ that defines the same 
presheaf on $\Sm_S$ as $E$.
\end{lemma}
\begin{proof}
Decomposition \eqref{eq:SmFrZFovZSH} of the functor $\Sm_S\to\SH(S)$ imples that $E$ defines a homotopy invariant linear framed additive presheaf.
Since the morphism $\sigma_X\in \Fr_1(X,X)$ for any $X\in \Sm_S$ goes to the canonical isomorphism $\PP_1/\infty\to T$ it follows that the framed presheaf $E$ is $\sigma$-stable.
\end{proof}

\begin{example}\label{ex:univStLinFr}
The representable presheaves $\Fr_*(-,Y)$ and $\ZF_*(-,Y)$ obviously are the universal examples of framed and linear framed presheaves.
The universal example of a homotopy invariant $\sigma$-stable linear framed presheaf 
is the presheaf $\ovZF(-,Y)$. 
\end{example}
\begin{example}
The examples of a $\sigma$-stable framed linear presheaves are given by $K$-theory, Chow groups, Hermitian $K$-theory, derived Witt-groups, Chow-Witt groups, algebraic cobordisms and so on. 
\end{example}

\subsection{$l_\varepsilon$-torsion presheaves}\label{subsect:LepsilonTorsion} 

The category $\Fr_+(S)$ and the category of framed presheaves relate as an approximation to the category $\SH(S)$.
We could say that the category the linearised $\ZF_*(S)$ linear framed presheaves
relates 
to the category of $\mathrm{H}\mathbb Z$-modules in $\SH(S)$, where $\mathrm{H}\mathbb Z$ is the image of the ring (topological) Eilenberg-MacLane spectrum. 
Now we are going to give such an approximation for
$\mathrm{H}\mathbb Z/l_\varepsilon$-modules, where 
\[l_\varepsilon=\sum_{i=0}^{l-1}\langle (-1)^i\rangle\in [\mathbb S,\mathbb S]_{\SH(S)},\]
and $\langle \lambda \rangle$ for an invertible $\lambda\in \calO^\times_S(S)$ is induced by the endomorphism of motivic sphere $\Gm^{\wedge 1}\to \Gm^{\wedge 1}$ given by the map \[\Gm\to \Gm\colon x\mapsto \lambda x.\]

\begin{definition}
For an invertible function $\lambda\in \mathcal O^\times(S)$, define 
$\langle \lambda\rangle \in \Fr_1^S(\pt_S,\pt_S)$ 
given by $(0_S,\A^1_S,\lambda x,p)$, where $p\colon \A^1_S\to \pt_S$ is the canonical projection.
\end{definition}
\begin{definition}
Define $h=\langle 1\rangle+\langle -1\rangle\in \ZF_1(pt_S,pt_S)$.
Define an element $l_{\varepsilon}\in \ZFr^S_1(pt_S,pt_S)$ by the formula $l_\varepsilon=nh$, for $l=2n$, $l_\varepsilon=nh+\langle 1\rangle$, for $l=2n+1$.
\end{definition}
\begin{remark}
For any $l\in \mathbb Z_{\geq 0}$,
$[l_\varepsilon] = [(Z(x^l),\A^1,x^l,p)]\in \ovZF_1(pt,pt)$, where 
$p\colon \A^1\to pt$ is the canonical projection. 
\end{remark}
\begin{definition}
We say that a linear framed presheaf $F$ is $l_\varepsilon$-torsion iff $l_\varepsilon\cdot a=0$, $\forall a\in F(X)$, $\forall X\in \Sm_S$, where $l_\varepsilon\cdot a$ denotes the image of $a$ under the map $F(X)\to F(X)$ induced by $l_\varepsilon$.
\end{definition}
\begin{remark}\label{rem:l-epsilonTorsionEinSH}
In view of Lemma \ref{lm:EovZFSHandstable} it follows that, 
any $l_\varepsilon$-torsion additive presheaf $E\colon \SH(S)\to \Ab^\op$ defines a 
homotopy invariant $l_\varepsilon$-torsion $\sigma$-stable linear framed presheaf $F\colon \ZF(S)\to \Ab^\op$.
In particular, for any  
an $l_\varepsilon$-torsion spectrum $\mathcal E\in \SH(S)$ the presheaves $E^{*,*}$ of the cohomology theory represented by $\mathcal E$ are of such type. 
\end{remark}

Moreover, any
$l_\varepsilon$-torsion functor 
$F\colon \ZF_*(S)\to \Ab^\op$ 
passes throw 
some (additive) 
category $\ZF/l_\varepsilon\ZF(S)$
\[\ZF_*(S)\to \ZF/l_\varepsilon\ZF(S)\to \Ab^\op.\]  
\begin{definition}\label{def:lepsilonFr_n}
For $\lambda\in \calO(S)^\times$ and $n\in \mathbb Z_{\geq 0}$, define the map
\[
c\mapsto c\bullet\langle\lambda\rangle \colon \Fr_n(X,Y)\to \Fr_{n+1}(X,Y) \colon 
(Z,\mathcal V, \varphi,g)\mapsto  (Z,\mathcal V, \lambda x_0,pr^*\varphi,g\circ pr). 
\]
For any $j=1,\dots n$ define the map
\[
c\mapsto c^\lambda_j \colon \Fr_n(X,Y)\to \Fr_n(X,Y) \colon  (Z,\mathcal V, \varphi,g)\mapsto  (Z,\mathcal V, \varphi_1, \dots\lambda\varphi_j, \dots\varphi_n,g),
\]
where $pr\colon \A^{n+1}_S=\A^1\times\A^n_S\to \A^n_S$, and $\varphi$ denotes the vector $(\varphi_1,\dots\varphi_n)$.
Define homomorphisms 
\[\begin{array}{ll}
(-) \bullet l_\varepsilon \colon \ZF_n(X,Y)\to \ZF_n(X,Y), & 
(-)^{l_\varepsilon}_j \colon \ZF_n(X,Y)\to \ZF_n(X,Y), \\
\,[c]\mapsto l_\varepsilon\bullet [c] = \sum_{i=0}^{l-1} ([c]\bullet \langle(-1)^i\rangle), &
[c]\mapsto [c]^{l_\varepsilon}_i=\sum_{i=0}^{l-1} [c]^{(-1)^i}_j.
\end{array}\]
%
\end{definition}

Denote by 
$l_\varepsilon \ZF_n(X,Y)\subset\ZF_n(X,Y),$ 
the sum of the images of the maps $(-)^{l_\varepsilon}_j$ for all $j$;
denote by $l_\varepsilon \ovZF_n(X,Y)$ the image of the map $l_\varepsilon \ZF_n(X,Y)\to  \ovZF_n(X,Y)$.
\begin{definition}
Define the additive graded category $\ovZF_*/l_\varepsilon(S)$ 
with objects being smooth schemes
and
morphisms being groups $\ovZF_*(X,Y)/l_\varepsilon \ovZF_*(X,Y)$.

\end{definition}

\begin{definition}\label{def:MatrixAction}
For a matrix $G\in \GL_n(X)$ and 
$f=(f_1,\dots f_n)\in \calO(X)^n$, for $X\in \Sch_S$, 
denote by $G\cdot (f_1,\dots f_n)\in \calO(X)^n$ the image of $f$ under the endomorphism on $\calO(X)^n$ defined by $G$.
For a set of invertible functions $\lambda_1,\dots \lambda_n\in \calO(X)^\times$ denote by $\langle\lambda_1,\dots \lambda_n\rangle\in \GL_n(X)$ the diagonal matrix; then $\langle\lambda_1,\dots \lambda_n\rangle\cdot (f_1,\dots f_n)= (\lambda_1f_1,\dots \lambda_nf_n)$.

Denote by $\mathbb Z\GL_n(X)$, $X\in \Sch_S$ the free abelian group of the set $\GL_n(X)$.
For any 
\[G=\sum_{i-1}^l m_iG_i\in \mathbb Z\GL_m(X),\] 
and 
\[c=(Z,\mathcal V,\varphi_1,\dots \varphi_n,g)\in \Fr_n(X,Y)\] 
we denote
\[G\cdot c=\sum_{i=1}^l m_i[(Z,\mathcal V, G_i(\varphi_1,\dots \varphi_n), g)\in \ZF_n(X,Y).\]

We write $G_0\stackrel{\A^1}{\sim} G_1$ for $G_0,G_1\in \mathbb Z\GL_n(X)$ iff there is $H\in \mathbb Z\GL_n(X\times\A^1)$ such that $i^*_0(H)=G_1$, $i^*_1(H)=G_1$, $i_0,i_1\colon X\to X\times\A^1$ denotes the zero and unit sections. 
\end{definition}

\begin{lemma}\label{lm:l_varepsilon(-)cdot}
For any $\lambda\in \calO(S)^\times$, $n\geq 0$, $j=1,\dots n$, 
and $c=(Z,\mathcal V, \varphi_1,\dots \varphi_{n},g)\in \Fr^S_{n}(X,Y)$ 
we have 
$[\langle \lambda\rangle\bullet c]=[\sigma c^\lambda_j]\in \ovZFr_{n+1}^S(X,Y)$,
where 
$\varphi_i^\prime=pr^*(\varphi_i)$, $i=1,\dots n$, $pr\colon \A^1\times\A^n_X\to \A^n_X$.
\end{lemma}
\begin{proof}
It follows by Lemmas \ref{lm:diagSLinEn} and \ref{lm:ElementaryMatHomotopy} in the Appendix that there is
 $$H\in \mathbb Z\GL_n(X\times\A^1), 
i^*_0(H)=[\langle \lambda ,1,1,\dots 1\rangle], i^*_1(H)=[\langle 1,\dots 1,  \lambda ,1,\dots 1\rangle]
,$$ with $\lambda$ at the $j$-th position (see Def. \ref{def:MatrixAction} for the notation above), 
$i_0,i_1\colon X\to X\times\A^1$.
Then 
$$i^*_0(\widetilde c) = [\langle \lambda\rangle c], i^*_1(\widetilde c)=[\sigma c^\lambda],$$
where $\widetilde c=H\cdot p^*(\langle \lambda\rangle c)\in \ZFr_n(X\times\A^1,Y)$, and $p\colon X\times\A^1\to X$.
\end{proof}

\begin{lemma}\label{lm:lvarepsilonTorsion}
Any homotopy invariant $l_\varepsilon$-torsion $\sigma$-stable linear framed presheaf is 
a presheaf on the quotient-category $\ovZF_*/l_\varepsilon \ovZF_*(S)$ such that $\sigma_X^*\colon F(X)\to F(X)$ is an isomorphism. 
\end{lemma}
\begin{proof}
Let $a\in F(Y)$, $Y\in \Sm_S$, and $c\in \Fr_n(X,Y)$.
Then by Lemma \ref{lm:l_varepsilon(-)cdot} $(l_\epsilon c)^*(a) = (l_\epsilon\bullet c)^*(a) = l_\epsilon\cdot c^*(a)=0$. So the claim follows.
\end{proof}

\section{Normal framing of a variety and inner framed correspondences.}
\label{subsect:FrVarInnerFr}

In this section we make some variations and generalisations of the notion of framed correspondences.
As mentioned in the previous section 
the framed correspondences 
describes morphisms of the form
\[X\times (\PP^1/\infty)^{\wedge l}\to Y\wedge (\A^1/(\A^1-0))^{\wedge l-d}\]
for $d=0$.
The case $d\geq 0$ leads to the notion 
of normal framings.
At the same time if we replace $(\PP^1/\infty)^{\wedge l}$ by a motivic space $\ovrelcurve/\relcurve_\infty$ for a projective scheme $ \ovrelcurve$ of pure relative dimension $l$ over $X$, and a closed subscheme $\relcurve_\infty$ of relative codimension one
\[(\ovrelcurve/\relcurve_\infty) \to Y\wedge (\A^1/(\A^1-0))^{\wedge l}, l=\dim_X\relcurve\]
we get the notion of $\relcurve$-inner framed correspondences, 
where $\relcurve=\ovrelcurve-\relcurve_\infty$.
This notions allows to split the construction of a framed correspondence into simpler steps. 

Let $S$ be a noetherian scheme of a finite Krull dimension.

%

\subsection{Normal framing and inner framed correspondences.}

\begin{definition}
\label{def:Zar_normal_framing}
\emph{A framed $S$-scheme} is a scheme $\relcurve$ over $S$ of pure relative dimension $d$
equipped with the data
$(i,W,\psi)$, where 
\begin{enumerate}
\item $i\colon \relcurve \to \A^{m}_S$ is a closed immersion;
\item $j\colon W\to \A^{m}_S$ is a Zariski open neighbourhood of $\relcurve$; 
\item and $\psi=(\psi_1,\dots \psi_{m-d})$ is
a vector of regular functions on $W$ such that 
$i(\relcurve)=Z(\psi)$ where $Z(\psi)$ stands for the common zero locus of $\psi_i$-s.
\end{enumerate}
The closed immersion and the set of functions $\psi_i$, $i=1,\dots {m-d}$,
is called as the \emph{level $m$ framing} of $\relcurve$.

Define the set $\lF_m^\Zar(\relcurve)$ of equivalence classes of 
level $m$ Zariski normal framings of $\relcurve$ up to the shrinking of the Zariski neighbourhood $W$ of $\relcurve$ in $\A^m$. 
So any framings $\Psi_1=(i_1,W_1,\psi^1)$ and $\Psi_2=(i_2,W_2,\psi^2)$ such that $W_2$ is Zariski neighbourhood of $\relcurve$ in $W_1$ and $\varphi_2=\varphi_1\big|_{W_2}$ are equal in $\lF_m^\Zar(\relcurve)$. 
\end{definition}
\begin{definition}\label{def:innerFrnr}
Let $\relcurve$ be an affine scheme over $S$ of a pure relative dimension $d$. 
Let $X, Y$ be $S$-schemes.
A \textit{$\relcurve$-inner normal framed correspondence} from 
$X$ to $Y$ 
over $S$ 
is a set of data $(Z, \mu, g)$ where 
\begin{enumerate}
\item  $Z \to \relcurve\times X$ is a closed immersion, $Z$ is finite over $X$, 
\item 
$\mu\colon \calO(Z)^d\simeq \calI/\calI^2$, where $\calI=\calI_{\relcurve\times X}(Z)\subset \calO(\relcurve\times X)$ in the vanishing ideal,
\item $g\colon Z\to Y$ the a regular map.
\end{enumerate}
The (pointed) set of normal $\relcurve$-inner framed correspondences from $X$ to $\relcurve$ is denoted by $\Fr_{\relcurve}^\mathrm{nr}(X,Y)$.
If $Y=\relcurve$ and $g$ equals to the canonical immersion 
we write $(Z,\mu)$ for $(Z,\mu,g)$ 
for simplicity.
\end{definition}

We are going to construct a map $\Fr_{\relcurve}^\mathrm{nr}(X,\relcurve)\times  \lF_m^\Zar(\relcurve)\to \ovZF_*(X,\relcurve)$.
Let $(Z,e)\in \Fr_{\relcurve}^\mathrm{nr}(X,\relcurve)$ and $(i,W,\psi)\in \lF_m^\Zar(\relcurve)$,
for $\relcurve\to S$ smooth of dimension $d$.

Consider the closed subscheme $Z(\calI^2)\subset \relcurve$,
where $\calI=\calI_{\relcurve\times X}(Z)\subset \calO(\relcurve\times X)$.
There is a unique vector of functions 
\[e=(e_1,\dots e_d) \in \calO_{Z(\calI^2)}(Z(\calI^2))^d, e\big|_Z=0, (\mathrm{d\,} e)\big|_Z=\mu,\] 
where $(\mathrm{d\,} e)\big|_Z$ denotes the differential of $e$ on $Z$.
Then $Z(e)=Z$. 

Consider the closed immersion given by the composition $Z\to\relcurve \to W\to \A^m_S$. 
Choose regular functions
\begin{equation}\label{eq:varphi=e}
\phi=(\varphi_1,\dots \varphi_d)\in \calO_W(W)^d, \varphi_i\big|_{Z(\calI^2)}=e_i, i=1,\dots d.
\end{equation} 
Then since $Z(\varphi\big|_{Z(\calI^2)})=Z$, 
it follows that 
$Z(\varphi\big|_{\relcurve})=Z\amalg \hat Z$,
and consequently 
\[Z(\varphi_1,\dots \varphi_d,\psi_1,\dots \psi_{m-d})=Z\amalg \hat Z\]
Thus there is a framed correspondence
\begin{equation}\label{eq:Phi((Z,mu),(i,W,psi))}\Phi=(Z, W, \varphi_1,\dots \varphi_d,\psi_1,\dots \psi_{m-d},g)\in \Fr_m(X,\relcurve),\end{equation}
where $g\colon Z\to \relcurve\times X\to \relcurve$.

\begin{lemma}\label{lm:[mu,psi]}
For any $(Z,\mu)$ and $(i,W,\psi)$ as above,
the class $[\Phi]\in \ovZF_*(X,\relcurve)$ is independent from the choice of $\phi$.
\end{lemma}
\begin{proof}
Let $\varphi,\varphi^\prime\in \calO_{W}(W)$ be two vectors satisfying \eqref{eq:varphi=e}, and 
$\Phi, \Phi^\prime\in \Fr_m(X,\relcurve)$.
Then the vector of functions $(1-\lambda)\varphi+\lambda \varphi^\prime\in \calO_{W\times\A^1}(W\times\A^1)$ there is an $\A^1$-homotopy connecting $\Phi$ and $\Phi^\prime$. Hence $[\Phi]=[\Phi^\prime]$.
\end{proof}

\begin{definition}
For any $c=(Z,\mu)\in \Fr_{\relcurve}^\mathrm{nr}(X,\relcurve)$ and $\Psi=(i,W,\psi)\in lF_m^\Zar(\relcurve)$ as above
denote \[[c,\Psi]=[\Phi]\in \ovZF_*(X,\relcurve)\] given by \eqref{eq:Phi((Z,mu),(i,W,psi))}.
This induces the map
\[[-,-]\colon \Fr_{\relcurve}^\mathrm{nr}(X,\relcurve)\times  \lF_m^\Zar(\relcurve)\to \ovZF_*(X,\relcurve).\]
\end{definition}


\begin{definition}\label{def:innerFrZar}
Let $\relcurve$ be an affine scheme over $S$ of a pure relative dimension $d$. 
Let $X, Y$ be $S$-schemes.
A \textit{$\relcurve$-inner Zariski framed correspondence} from 
$X$ to $Y$ 
over $S$ 
is a set of data $(Z, V, \phi, g)$ where 
\begin{enumerate}
\item  $Z \to \relcurve\times X$ is a closed immersion, $Z$ is finite over $X$, 
\item $V\to \relcurve$ is Zariski open neighbourhood of $Z$,
\item 
$\phi=(\varphi_1,\dots \varphi_d)\in \calO(V)^d$, $Z(\phi)\cong Z$,
\item $g\colon Z\to Y$ is a regular map.
\end{enumerate}
Denote by $\Fr_{\relcurve}^\Zar(X,Y)$ the set of equivalence classes of $(Z,V,\phi,g)$ as above 
up to shrinking of the Zariski neighbourhood $V$ of $Z$.
\end{definition}
Any correspondence
$c=(Z,V,\phi)\in \Fr_{\relcurve}^\Zar(X,\relcurve)$ defines an element 
$(Z,\mu)\in \Fr_{\relcurve}^\mathrm{nr}(X,\relcurve)$, where $\mu$ is given by the differential $\mathrm{d_Z\,}\phi$ of $\phi$ on $Z$ in $\relcurve$.
%

\begin{definition}\label{def:ZF_C-and-lepsilonFr_C}
Under the assumptions of Def. \ref{def:innerFrZar} 
and in a similar way to Def. \ref{def:ZF}, and \ref{def:ovZF},
define 
the presheaves on $\Sm_S$
\[\ZF_\relcurve^\mathrm{Zar}(\relcurve)=\ZF_\relcurve^\mathrm{Zar}(-,\relcurve),\,
\ovZF_\relcurve^\mathrm{Zar}(\relcurve)=\ovZF_\relcurve^\mathrm{Zar}(-,\relcurve)
\]
In a similar way to Def. \ref{def:lepsilonFr_n} we define 
the subpresheaves
$l_\varepsilon\ZF_\relcurve^\mathrm{Zar}(-,\relcurve)$, and
$l_\varepsilon\ovZF_\relcurve^\mathrm{Zar}(-,\relcurve)$,
and the quotient-presheaves
$\ZF_{\relcurve}^\Zar/l_{\varepsilon}(-,\relcurve)$, and
$\ovZF_{\relcurve}^\Zar/l_{\varepsilon}(-,\relcurve)$.
\end{definition}
%
Denote by $\mathbb Z\lF_m^\Zar(\relcurve)$ the free abelian group associated with $\lF_m^\Zar(\relcurve)$.
Then we get linear homomorphisms
\[[-,-]\colon \ovZF_{\relcurve}^\Zar(X,\relcurve)\otimes  \mathbb Z\lF_m^\Zar(\relcurve)\to \ovZF_*(X,\relcurve).\]
\[[-,-]\colon \ovZF_{\relcurve}^\Zar/l_{\varepsilon}(-,\relcurve)\otimes \mathbb Z\lF_m^\Zar(\relcurve)\to \ovZF_*/l_\varepsilon(-,\relcurve).\]

%

\subsection{Framed maps.}

In this subsection we consider the following special type of framed correspondences, which could be called as framed maps.

\begin{definition}
Let $f\colon X\to Y$ be a morphism of $S$-schemes. 
For any $C\in \GL_n(X)$ 
define the framed correspondence of the level $n$
\[
f^C = (0_X, C^*(t_1,t_2, \dots t_n), f\circ \pr_X) \in \Fr_n(X,X)\]
where $C^*(t_1,t_2, \dots t_n)$ denote the inverse image of the coordinate functions along the induced morphism $C^*\colon \A^n_X\to \A^n_X$, 
and $\pr_X\colon (\A^n_X)^h_{0\times X}\to X$ is the projection.

For any morphism $s\colon X\to Y$
define $\sigma_n s^\lambda=\sigma^{n,\lambda}_Y\circ s=s\circ \sigma^{n,\lambda}_X\in \Fr_n(X,Y)$.

\end{definition}

For any short exact sequence
\[N_1\hookrightarrow N\twoheadrightarrow N_2\]
and isomorphisms
\[\mu\colon N_2\simeq \mathbf{1}^{n_1}_S, \nu\colon N_1\simeq \mathbf{1}^{n_2}_S\]
we have an element 
\begin{equation*}
[\mu,nu]\in [N,\mathbf{1}^n)S]_E
\end{equation*}
in the set of classes of isomorphisms $\Iso(N,\mathbf{1}^n)S)$ with respect to the elementary transformations, see Definition \ref{def:[mu,nu]}. 

\begin{lemma}\label{lm:inframingJac}
Let $\Gamma=(Z,\mu,g)\in \Fr^\nr_\relcurve(S,Y)$, $\Psi\in \lF^\Zar_{N}(\relcurve)$.
Then $[\Gamma,\Psi] =[g\circ s]^{[\mu,d\psi]} \in \ovZF_N(S,Y)$. 
\end{lemma}
\begin{proof}

Consider the closed immersions $Z\to \relcurve\to \A^n_X$ given by $\Gamma$ and $\Psi$.
Since $Z$ is affine there is a splitting
\[\Omega_{Z/\A^N_S}\simeq i^*(\Omega_{\relcurve/\A^N_S})\oplus \Omega_{Z/\relcurve}.\]
Hence the trivialisation 
\[\mu\to \mu\colon i^*(\Omega_{\relcurve/\A^N_S})\simeq \mathbf{1}^{d}_Z\] 
lifts to a morphism
\[\mu^\prime\colon \Omega_{Z/\A^N_S}\to \mathbf{1}^d_Z.\]
Let $\phi=(\phi_1,\dots \phi_d)\in \calO_{\A^N_S}(\A^N_S)^d$ be a vector of regular functions such that  
$\mathrm{d}_Z\, \phi = \mu^\prime$.
Then $[\mathrm{d}_Z\, (\phi,\psi)]=[\mu,d\psi]$.
Hence
\[[\Gamma,\Psi] = [(Z,\phi,\psi,g^\prime)]=[g\circ s]^{[\mu,\mathrm{d_Z}\, \psi]}.\]
\end{proof}

\begin{lemma}\label{lm:s0C0s1C1}
Let $i\colon Z\to X$ be a closed immersion that defines an affine henselian pair. 
Let $s_0,s_1\colon X\to Y$ be regular maps;
let $C_0, C_1\in \GL_n(X)$, $i^*(C_0)=i^*(C_1)$.

If $[s_0^{C_0}]=[s_1^{C_1}]\in \ovZF/_{l_\varepsilon}(X,Y)$, then
\[
[s_0]=[s_1]\in \ovZF/_{l_\varepsilon}(S,Y)
.\]
\end{lemma}
\begin{proof}
The claim follows by the sequence 
\[[s_0]=[\sigma^{C_0^{-1}}\circ s_0^{C_0}]=[\sigma^{C_0^{-1}}\circ s_1^{C_1}]=[s_1],\]
where the last equality follows by Corollary \ref{cor:henspairElemHomdiagdetectinvrad} and Lemma \ref{lm:ElementaryMatHomotopy}.\end{proof}

%

\section{Fine compactification}\label{subsect:FineCompact}

\begin{definition}
\label{def:fine_compactification}
Let $S=\Spec R$ be the spectrum of a ring, $\relcurve$ be a normal scheme, and $\relcurve\to S$ be 
a 
morphism of relative dimension $1$. 
We call by a \textit{fine compactification of $\relcurve$ over $S$} a set of data
$(j\colon \relcurve\hookrightarrow\relpcurve,\struct(1), \zeta_\infty, \zeta_c)$ with $\relpcurve$ being a projective scheme over $U$ of (pure) relative dimension one, 
$j$ being an open dense immersion over $U$,
$\struct(1)$ being a very ample line bundle over $\relpcurve$,
and $\zeta_\infty\in \Gamma(\relpcurve,\struct(1))$ and $\zeta_c\in \Gamma(\relpcurve,\struct(1))$ are such that 
\begin{enumerate}
\item
$\relcurve=\relpcurve-Z(\zeta_\infty)$;
\item
$Z(\zeta_\infty)$ is finite over $S$.
\item
$Z(\zeta_c)$ is finite over $S$, $Z(\zeta_c)\cap Z(\zeta_\infty)=\emptyset$.
\end{enumerate}
\end{definition}
\begin{remark}
In difference to \cite[Definition 5.4]{AD-Rigid} 
we require the third additional property of the section $\zeta_c$. 
In the local base case considered in \cite{AD-Rigid} such $\zeta_c$ always exists.
\end{remark}
\begin{lemma}\label{lm:FineCompCahceZetac}
Let $\relcurve$ be a 
scheme of relative dimension one over an affine noetherian 
base $S=\Spec R$,
and $(\ovrelcurve, \calO(1), \zeta_\infty, \zeta_c)$ be a fine compactification.
Let $\Gamma\subset\relcurve $ be a closed subscheme finite over $U$.
Then there is a fine compactification 
$(\ovrelcurve, \calO^\prime(1), \zeta_\infty^\prime, \zeta_c^\prime)$ of $\relcurve$
such that $\zeta_c^\prime\cap \Gamma=\emptyset$.
\end{lemma}
\begin{proof}
By Serre's theorem (see Prop. \ref{prop:CorofSerreTh}) for a large enough $b$ there is a section
$\zeta^\prime_c\in \Gamma(\ovrelcurve, \calO(b))$, 
\[\zeta_c^\prime\big|_{\Gamma}=\zeta_\infty^b, 
\zeta_c^\prime\big|_{Z(\zeta_\infty)}=\zeta_c^b.\]
Then
$(\ovrelcurve, \calO(b), \zeta_\infty^b, \zeta_c^\prime)$
is the required compactification. 
\end{proof}

\begin{lemma}\label{lm:l-epsDivision}
Let $\relcurve$ be a scheme of relative dimension one over a base $S=\Spec R$ with a fine compactification 
$(\ovrelcurve, \calO(1), \zeta_\infty,\zeta_c)$.
Let $c=(Z(\varphi),\relcurve,\varphi)\in \Fr_{\relcurve}(U,\relcurve)$,
and assume that $\varphi=u (r/\zeta_\infty)^l$, $u\in \calO^\times(\relcurve)$, $r\in \Gamma(\ovrelcurve, \calO(1))$.
Then $[c]\in l_\varepsilon \ZF_{\relcurve}(U,\relcurve)$.
\end{lemma}
\begin{proof}
Consider the sequence of equalities in $\ovZF_{\relcurve}(S,\relcurve)$
\begin{multline*}
[(Z(r^l),\relcurve, ur^l)]\stackrel{h}{\sim} 
[(Z(r^{l-1}(r+\zeta_\infty)),\relcurve, u r^{l-1}(r+\zeta_\infty)/\zeta_\infty^{l} )] =  \\
[(Z(r^{l-1}),\relcurve, u r^{l-1}/\zeta_\infty^{l-1})] + [(Z(r+\zeta_\infty),\relcurve, (-1)^{l-1}u(r+\zeta_\infty)/\zeta_\infty )]
\stackrel{g}{\sim} \\
[(Z(r^{l-1}),\relcurve, u r^{l-1}\zeta_\infty^{l-1} )] + [(Z(r),\relcurve, (-1)^{l-1}ur/\zeta_\infty)] = \\
[(Z(r^{l-1}),\relcurve, u r^{l-1}/\zeta_\infty^{l-1})] + (-1)^{l-1}[(Z(r),\relcurve, ur/\zeta_\infty)],
\end{multline*}
where the homotopies $h$ and $g$ are given by the correspondences
$$\begin{array}{ll}
h = &(Z(r^{l-1}(r+\lambda+\zeta_\infty)), \relcurve, u r^{l-1}(r+\lambda\zeta_\infty)/\zeta_\infty^{l} ), \\
g=& (Z(r+\lambda\zeta_\infty), \relcurve, (-1)^{l-1}u (r+\lambda\zeta_\infty)/\zeta_\infty ) .
\end{array}$$
The fact that the supports $Z(r^{l-1}(r+\lambda+\zeta_\infty))$ and $Z(r+\lambda\zeta_\infty)$ are finite over $S\times\A^1$
follows from the fact that $Z(r)\cap Z(\zeta_\infty)=\emptyset$.

Then by induction we get 
\begin{multline*}
c=[(Z(r^l),\relcurve, ur^l/\zeta_\infty^{-1})] = \sum\limits_{i=1}^l (-1)^{i-1}[(Z(r),\relcurve, ur/\zeta_\infty)]=\\
l_\varepsilon [(Z(r),\relcurve, ur/\zeta_\infty)] \in \ovZF_{\relcurve}(S,\relcurve).\end{multline*}
\end{proof}

\section{Framed $\A^1$-homotopy in relative curve.}\label{sect:SectRelCurve} 

Given an affine henselian pair $i\colon Z\hookrightarrow U$, 
and a flat morphism of pure dimension one $\relcurve\to U$ admitting a fine compactification (Def. \ref{def:fine_compactification}), and Zariski normal framing (Def. \ref{def:Zar_normal_framing}).
In this section we prove 
that 
framed maps from $U$ to $\relcurve$
that are equal over $Z$ define the same class in 
$\ovZF(U,\relcurve)$.
The main result is Theorem \ref{th:Curve}.

We start with the case of $\relcurve$-inner-correspondences $\ovZF_{\relcurve}(U,\relcurve)$
in sense of Definition \ref{def:innerFrnr}
for a flat morphism of pure dimension one $\relcurve\to U$ that admits a fine compactification.
Let $\gamma_0,\gamma_1\colon U\to \relcurve$ be a pair of $U$-maps 
such that 
$\gamma_0\big|_Z=\gamma_1\big|_Z$, 
and
$\relcurve\to U$ is smooth over 
the subscheme $S=\gamma_0(Z)=\gamma_1(Z)$ in $\relcurve$.
Denote 
the closed subschemes $S_i=\gamma_i(Z)$, and 
ideals $I_i=I(S_i)\subset \calO_{\relcurve}(\relcurve)$, for $i=0,1$.
Let
\begin{equation}\label{eq:muidualN(Si)O(Si)}
\mu_i\colon \calO_{S_i}(S_i)\simeq I_i/I_i^2, i=0,1.
\end{equation}
and suppose \[i^*(\mu_0)=i^*(\mu_1)=\mu,\] 
where $i^*(-)$ denotes the base change. 

\begin{lemma}\label{lm:HomotFunct}
%
%
Under the notation above for 
any $l\in \mathbb Z$ invertible on the scheme $U$
the is the equality
\[[\gamma_0]^{\mu_0}=[\gamma_0]^{\mu_0}\in \ZF_\relcurve^\nr/l_\varepsilon(U,\relcurve).\]
\end{lemma}
\begin{proof}
Denote $C=\relcurve\times_U Z$, $\ovC=\ovrelcurve\times_U Z$.
To prove the claim we construct a $\relcurve$-inner Zariski framed correspondence
$h=(Z,\relcurve\times\A^1,\varphi)\in \ZF^\mathrm{Zar}_{\relcurve}(U\times\A^1,\relcurve)$
such that
\[[h\circ i_0]=[\gamma_0]^{\mu_0}, [h\circ i_1]=[\gamma_1]^{\mu_1}\]
in $\ZF_\relcurve^\Zar/l_\varepsilon(U,\relcurve)$.
%
The function
$\varphi\in \calO_{\relcurve\times\A^1}(\relcurve\times\A^1)$
is such that 
\begin{itemize}
\item[(1)] 
$Z(\varphi)$ is finite over $U\times\A^1$,
\item[(2)] 
the restriction $\varphi\big|_{S_0\times \{0\}\cup S_1\times \{1\}}$ vanishes, and
\[\begin{array}{lll}
\varphi_0\big|_{Z(I^2(S_0))}&=\mu_0(1)& \in \calO_{\relcurve}(\relcurve)/I(S_0)^2\\
\varphi_1\big|_{Z(I^2(S_1))}&=\mu_1(1)& \in \calO_{\relcurve}(\relcurve)/I(S_1)^2
,\end{array}\]
where \[\varphi_0=\varphi\big|_{\relcurve\times 0}, \varphi_1=\varphi\big|_{\relcurve\times 1};\]
\item[(3)] 
the framed correspondences
\[
[(\relcurve -S_0, Z(\varphi_0)-S_0, \varphi_0)], 
[(\relcurve -S_1, Z(\varphi_1)-S_1, \varphi_1)]
\] are in 
$l_\varepsilon \ovZF^\mathrm{Zar}_{\relcurve}(U,\relcurve)$.
\end{itemize}

The required function $\varphi$ is given by
\[\varphi=(\tau_0(s_0 r_0^l)(1-\lambda)+\tau_1(s_1 r_1^l)\lambda)/d ,\]
where $\tau_0$, $\tau_1$, $s_0$, $s_1$, $r_0$, $r_1$, and $d$ are constructed in Lemmas \ref{lm:LR0LR1}, \ref{lm:Dsect}, \ref{lm:rzrz} that follow in the text.
Then points (1) and (2) above follows immediate from the properties claimed in the mentioned lemmas.
Point (3) in the list above 
follows by Lemma \ref{lm:l-epsDivision} applied 
to the curves $\relcurve-(D\cup S_i)$, $i=0,1$, 
equipped with the fine compactifications $(\ovrelcurve, \calO(k), z_i, \zeta_c^k)$,
and the correspondences
\[[(\relcurve -S_i, Z(\varphi_i)-S_i, \varphi_i)]=[(\relcurve -S_i, Z(\varphi_i)-S_i, (r_i/z_i)^l (s_i/(z_i \cdot d)))]\]
in $\ZF_{\relcurve}(U,\relcurve)$.
\end{proof}
To complete the proof we proceed with used Lemmas \ref{lm:LR0LR1}, \ref{lm:Dsect}, \ref{lm:rzrz}. 
Denote 
\begin{gather*}
W_0=Z(I^2(S_0))\subset \relcurve, W_1=Z(I^2(S_1))\subset \relcurve, \\
W=W_0\times_U Z=W_1\times_U Z\subset C.
\end{gather*}

\begin{lemma}\label{lm:LR0LR1}
There are line bundles 
$\calL_0, \calL_1,\calL, \calR_0, \calR_1$ on $\ovrelcurve$
with sections 
\begin{gather}
\label{eq:s0L0S0}
s_0\in\Gamma(\ovrelcurve, \calL_0), Z(s_0)=S_0,\\
\label{eq:s1L1S1}
s_1\in\Gamma(\ovrelcurve, \calL_1), Z(s_1)=S_1,
\end{gather}
and 
$\varepsilon \in \Gamma( W, \calL_0^\times )$,
and isomorphisms
\begin{gather}\label{eq:LRLR}
\tau_0\colon \calL_0\otimes\calR_0^l\simeq \calL,
\tau_1\colon \calL_1\otimes\calR_1^l\simeq\calL,\\
\label{eq:xinu}
\xi\colon i^*(\calL_0)\simeq i^*(\calL_1), \nu_0\colon i^*(\calR_0)\simeq \calO(\ovC), \nu_1\colon i^*(\calR_1)\simeq \calO(\ovC)
\end{gather}
that are agreed on the fibre over $Z$, and such that
\begin{gather}
\label{eq:s0s1C}\xi(i^*(s_0))=i^*(s_1),\\
\label{eq:ssTT}
i^*(s_0\big|_{W_0})/\varepsilon = i^*(s_1\big|_{W_1})/\xi(\varepsilon) = i^*(\mu_0(1))=i^*(\mu_1(1)).
\end{gather}
\end{lemma}
\begin{proof}
Define $\calL_0=\calL(S_0)$, $\calL_1=\calL(S_1)$.
Since $S_0\times_U Z=S_1\times_U Z$, \[i^*(\calL_0)\simeq i^*(\calL_1).\]
So $i^*(\calL_0\otimes\calL_1^{-1})\simeq\calO(C)$, and by Proposition \ref{prop:lrootLineBun}
there is a line bundle $\calR$ on $\relcurve$ such that \[\calR^l\simeq \calL_0\otimes\calL_1^{-1}, i^*(\calR)\simeq\calO(C).\]
Put $\calR_0=\calO(\relcurve)$, $\calR_1=\calR$, $\calL=\calL_0\otimes\calR^l$. 
Then isomorphisms $\tau_0$ and $\tau_1$ in \eqref{eq:LRLR}, and $\nu_0$, and $\nu_1$ in \eqref{eq:xinu} holds by construction.
Now we set \[\xi=(\tau_1^{-1}\circ\tau_0)\otimes\nu_0^{\otimes -1}\otimes(\nu_1^{-1})^{\otimes -1}.\] 

By construction there are some $s_0\in \Gamma(\relcurve, \calL_0)$, $Z(s_0)=S_0$, $s_1\in \Gamma(\relcurve, \calL_1)$, $Z(s_1)=S_1$. Since $S_0\times_U Z=S_1\times_U Z$, the quotient section $\xi(i^*(s_0))/i^*(s_1)\in \Gamma(\ovC,\calO)$ is invertible. Then \[\xi(i^*(s_0))/i^*(s_1)=p_Z^*(\lambda)\] for some $\lambda\in \Gamma(Z,\calO^\times)$,
since the morphism $p_Z\colon \ovC\to Z$ is projective. Choose any lift $\widetilde\lambda\in \Gamma(U,\calO^\times)$, and redenote $s_1\colon=\widetilde\lambda s_1$. Then we get \eqref{eq:s0s1C}.

Since $Z(s_0)=S_0=Z(T_0)$, and $Z(s_1)=S_1=Z(T_1)$ the quotient sections $i^*(s_0\big|_{W_0})/i^*(T_0)$ and $i^*(s_1\big|_{W_1})/i^*(T_1)$ are invertible on $W_0$ and $W_1$ respectively. Equalities \eqref{eq:ssTT} follows, since $i^*(T_0)=i^*(T_1)$ and because of \eqref{eq:s0s1C}.
\end{proof}

\begin{lemma}\label{lm:Dsect}
Under the assumptions of Lemma \ref{lm:HomotFunct} and the claim of Lemma \ref{lm:LR0LR1}
there are 
a fine compactification $(\ovrelcurve,\calO(1), \zeta_\infty, \zeta_c)$ of $\relcurve$,
a reduced closed subscheme $D\subset \ovrelcurve$ finite and surjective over $U$,
and an integers $K_d\in \mathbb Z$, for all $b\in \mathbb Z$, such that 
\[
D\cap S_0=D\cap S_1=D\cap Z(\zeta_c)=\emptyset,
D\supset Z(\zeta_\infty)_{red},
\]
and for any $k>{K}_d$, 
$\exists d\in \Gamma(\ovrelcurve,\calL(lbk))$ such that 
\[
Z(d)_{red}=D,
i^*(d)\big|_{W}=\tau_0(\varepsilon \cdot \nu_0^l(1)) \cdot \zeta_\infty^{lbk}.
\]
\end{lemma}
\begin{proof}
By \eqref{eq:s0L0S0} 
$\calL_0\simeq \calL(S_0)$, 
and by \eqref{eq:muidualN(Si)O(Si)}
the normal bundle $N_{S_0/\relcurve}$ is trivial.
Hence $\calL_0\big|_{S_0}$ is trivial.
Equalities \eqref{eq:xinu} imply that $\calR_0\big|_{W}$ is trivial. 
Thus $\calL\big|_{W}$ is trivial. 
Let $u\in \Gamma(W,\calL^\times)$ be an invertible section.
%

By assumption 
$\relcurve$ admits a fine compactification,
and because of Lemma \ref{lm:FineCompCahceZetac} we can assume that 
$Z(\zeta_c)\cap (S_0\cup S_1)=\emptyset$.
%
Then since $\calO(1)$ is ample, 
by Serre's theorem on ample bundles, see Prop. \ref{prop:CorofSerreTh}, 
for a large enough $b\in \mathbb Z$ there is a section 
\[
d^+\in \Gamma(\ovrelcurve,\calO(lb)), 
d^+\big|_{W}=\zeta_\infty^{lb}, 
d^+\big|_{Z(\zeta_c)}=\zeta_\infty^{lb},
\]
and
for a large enough $K_0\in \mathbb Z$ there is a section 
\[
d_0\in \Gamma(\ovrelcurve,\calL(lbK_0)), 
d_0\big|_{Z(\zeta_\infty)}=0, d_0\big|_{W}=u\zeta_\infty^{lbK_0}, 
d_0\big|_{Z(\zeta_c)}=s_0\cdot \nu_0(1)^l\cdot \zeta_\infty^{lb}.
\]

Set $D=Z(d_0 \cdot d^+)_{red}$. 
For all $k>{K_0}+1$, we set $d=d_0\cdot (d^+)^{k-K_0}$.
\end{proof}

\begin{lemma}\label{lm:rzrz}
Assume the claim of Lemmas \ref{lm:LR0LR1} and \ref{lm:Dsect}.
Then there is an integer ${K}_r$ such that
$\forall k>{K}_r$
$\exists r_0,z_0\in \Gamma(\ovrelcurve,\calR_0(k)),r_1,z_1\in \Gamma(\ovrelcurve,\calR_1(k))$
\begin{gather*}
\begin{array}{ll}
z_0\big|_{S_0\cup D}=0, & z_1\big|_{S_1\cup D}=0, \\
r_0^l\big|_{W}=\nu_0(1)\zeta_\infty^k, & r_1^l\big|_{W}=\nu_1(1)\zeta_\infty^k, \\
Z(r_0)\cap D=\emptyset, & Z(r_1)\cap D=\emptyset,\\
Z(r_0)\cap Z(z_0)=\emptyset, & Z(r_1)\cap Z(z_1)=\emptyset,
\end{array}
\\
(s_0r_0^l)\big|_D=(s_1r_1^l)\big|_D.
\end{gather*}
\end{lemma}
\begin{proof}
By Serre's theorem (see Prop. \ref{prop:CorofSerreTh}) there is $K_0$ such that for all $k>K_0$ there are section
\[\begin{array}{ll}
r_0\in \Gamma(\ovrelcurve,\calR_0(k)), & r_0\big|_{W}=\nu_0(1)\zeta_\infty^k, r_0\big|_{D}=\nu_0(1)\zeta_c^k\\
\end{array}\]
where $\zeta_\infty$ and $\zeta_c$ are given by the fine compactification provided by the claim of Lemma \ref{lm:Dsect}.
Then by Serre's theorem, Prop. \ref{prop:CorofSerreTh}, again there is ${K}_r$ such that for all $k>K_r$ there are sections $r_0$ and $r_1$ as above and sections
$$\begin{array}{ll}
r_1\in \Gamma(\ovrelcurve,\calR_1(k)), & r_1\big|_{W}=\nu_1(1)\zeta_\infty^k, r_1\big|_{D}=(s_0\big|_D/s_1\big|_D)\cdot r_0\big|_D ,\\
z_0\in \Gamma(\ovrelcurve,\calR(k)), & z_0\big|_{D\cup S_0}=0 , z_0\big|_{Z(r_0)}=\nu_0(1)\zeta_\infty^k\\
z_1\in \Gamma(\ovrelcurve,\calR(k)), & z_1\big|_{D\cup S_1}=0 , z_1\big|_{Z(r_1)}=\nu_1(1)\zeta_\infty^k .
\end{array}$$
So the claim follows. 
\end{proof}

This finishes the proof of Lemma \ref{lm:HomotFunct} on $\relcurve$-inner framed correspondence.
Now we deduce the result used in the next section.
\begin{theorem}\label{th:Curve}
Let $i\colon Z\hookrightarrow U$ be an affine henselian pair.
Let $\mathcal C\to U$ be a flat morphism of pure dimension one 
admitting a level $n$ Zariski normal framing (Def. \ref{def:Zar_normal_framing}) 
and a fine compactification (Def. \ref{def:fine_compactification}).
Let $\gamma_0,\gamma_1\colon U\to \mathcal C$ be morphisms of $U$-schemes such that 
$i^*(\gamma_0)=i^*(\gamma_1)$ 
and such that the morphism $\mathcal C\to U$ is smooth over $\gamma_0(Z)=\gamma_1(Z)$.
Let $e\colon \mathcal C\to Y$ be a morphism of $U$-schemes with $Y$ being smooth affine over $U$.

Then for every $l\in \mathbb{Z}_{>0}$ such that $l\in \calO_U(U)^\times$ 
\[
\sigma^{n}[e\circ \gamma_0]=\sigma^{n}[e\circ \gamma_1]\in\ovZF_{n}/l_\varepsilon(U,Y).
\]
\end{theorem}
\begin{proof}

Denote 
$S_0=\gamma_0(U)$, $S_1=\gamma_1(U)$. 
Denote 
$C=\relcurve\times_U Z$, and 
$S=S_0\times_U Z=S_1\times_U Z\subset C$. 

Let $\Psi=(\mathcal V,\psi_1,\dots \psi_{n-1},g)$ be a normal framing of $\relcurve$. 
By assumption $\relcurve\to U$ is smooth over the subscheme $S$. 
Then the differentials $\mathrm{d\,}\phi_i$ define the trivialisation of the conormal bundle $\Omega_{\relcurve/\A^n_U}$. 
Since the relative tangent bundle of $\A^n_U$ over $U$ is trivial, and 
$\relcurve$ is one-dimensional subscheme,
it follows that the tangent bundle $T_{\relcurve}$ is trivial on $S$. 
Consequently, $T_{\relcurve/U}$ is trivial on
$S_0=\gamma_0(U)$ and $S_1=\gamma_1(U)$.

Choose a trivialisation $\mu\colon T_{\relcurve/U}\big|_{S}\simeq \mathbf 1_S$, 
and lifts
\[\mu_1\colon T_{\relcurve/U}\big|_{S_1}\simeq \mathbf 1_{S_1}, \mu_0\colon T_{\relcurve/U}\big|_{S_0}\simeq \mathbf 1_{S_0}.\]

Now it follows by 
Lemma \ref{lm:HomotFunct} that 
\[[\gamma_0]^{\mu_0}=[\gamma_1]^{\mu_1}\in \ZF_{\relcurve}/l_\varepsilon(U,\relcurve).\]
Hence by 
Lemma \ref{lm:inframingJac} 
\[[\gamma_0]^{[\mathrm{d\,}\psi,\mu_0]}=[\gamma_1]^{[\mathrm{d\,}\psi,\mu_1]}\in \ZF_n/l_\varepsilon(U,\relcurve).\]
where $\psi=(\psi_1,\dots \psi_{n-1})$.
Thus
by Lemma \ref{lm:s0C0s1C1}
\[\sigma^n[\gamma_0]=\sigma_n[\gamma_1]\in \ZF_n/l_\varepsilon(U,Y),\]
since $i^*(\mu_0)=\mu=i^*(\mu_1)$.
\end{proof}

One can see that assumptions in the above theorem are stronger then in 
Lemma \ref{lm:HomotFunct},
since we ask in addition 
a Zariski normal framing for $\relcurve$ over $U$.
Actually, by the following remark, this extra assumption is not restrictive.
At the same time in the next section we apply the theorem to the relative curve given with the normal framing.
\begin{remark}
Let $\relcurve$ be an affine $U$-scheme of a pure dimension one that admits a fine compactification 
over an affine scheme $U$.
For any closed subschemes $S_i\subset \relcurve$, $S_i\simeq S$, $i=0,1$, 
such that
the normal bundles $N_{S_i/\relcurve}$ are trivial,
there is a Zariski neighbourhood $\relcurve^\prime\supset S_0\cup S_1$ that has 
a fine compactification and a normal framing over $S$ in the same time.
\end{remark}

\section{The rigidity theorem (the main result).}\label{sect:ContractingDeduct}

In the section we prove 
rigidity Theorem \ref{th:introduction:Rigidity:framed} (see Theorem \ref{th:Rigidity})
using the result of Section \ref{sect:SectRelCurve}.
This 
is the main novel content and the central part of the article.

Throughout the section we work
with given 
$l_\varepsilon$-torsion homotopy invariant $\sigma$-stable linear framed presheaf 
$F\colon Sm_k\to \Ab^\op$ 
for some $l\in \mathbb Z$, $l\in k^\times$, and a base field $k$.
Namely, we prove the isomorphism 
\[F(U)\simeq F(Z)\]
for an affine smooth henselian pair $Z\hookrightarrow U$ over $k$.


\subsection{The case $T_Z\simeq \mathbf{1}^d_Z$, $N_{Z/U}\simeq \mathbf 1_Z$.}

\begin{lemma}\label{lm:StabNtriv}
Given a closed immersion $i\colon X\to Y$ in $\SmAff_S$ over an affine base scheme $S$.
Assume that schemes $X$ and $Y$ are affine.
Assume that $T_{X/S}\simeq \mathbf 1^n_X$, and $T_{Y/S}\simeq \mathbf 1^m_Y$.
%
Then the normal bundle $N_{X/Y}$ is stably trivial.
\end{lemma}
\begin{proof}
The claim follows, since 
\[N_{X/Y}\oplus \mathbf 1_X^n\simeq N_{X/Y}\oplus T_{X/S}
\stackrel{Lm \ref{lm:TTNsplitting}}{\simeq} i^*(T_{Y/S})=\mathbf 1^m_X.\]
\end{proof}

Let $i\colon Z\hookrightarrow X\in \Sm_k$ be a closed immersion of smooth affine schemes, 
and $r\colon X\to Z$ be a smooth retraction, i.e. $r$ is a smooth morphism such that $r\circ i=\id_{Z}$.
Suppose that $X$ is of pure dimension $d$ over $k$, $\codim_X Z=1$; 
suppose that the tangent bundle $T_Z$, 
and the (relative) tangent bundle $T_{X/Z}$ (with respect to $r$) are trivial. 
In particular, it follows that the tangent bundle $T_X$ and the normal bundle $N_{Z/X}$ are trivial.
Denote by \[\mathrm{can}\colon U=X^h_Z\to X\] the canonical morphism.

\begin{lemma}\label{lm:dimd}
Suppose $i\colon Z\hookrightarrow X$ is a closed immersion of smooth affine schemes over a field $k$.
Let $U=X^h_Z$. 
Then there is \begin{itemize}
\item[(0)]
a smooth $U$-scheme $\calX$ of relative dimension $d$, and 
a smooth closed subscheme $\calE$ of relative dimension one over $U$,
such that the normal bundle $N_{\calE/\calX}$ is trivial;
\item[(1)] a Zariski normal framing of $\calX$ over $U$; 
\item[(2)]
a closed immersion 
$\Delta_Z\hookrightarrow \calE_Z=\calE\times_U Z$ such that
the normal bundle $N_{\Delta_Z/\calE_Z}$ is trivial;
\item[(3)]
$U$-morphisms $\gamma,\delta\colon U\to \calE$, and a $k$-morphism $g\colon \calX\to X$
such that 
\[\gamma(Z)=\delta(Z)=\Delta_Z,\] and
\[g\big|_{\calE}\circ r_0=\mathrm{can}\colon U\to X, g\big|_{\calE}\circ r_1=\mathrm{can}\circ r;\]
\item[(4)]
a projective $U$-scheme $\ovcalX$, and 
an open immersion $\calX\to \ovcalX$ such that
\[\calX_\infty= \ovcalX-\calX = X_\infty\times U\] for some (projective) $k$-scheme $X_\infty$ of dimension $d-1$.
\end{itemize}\end{lemma}
\begin{proof}
Since $X$ is affine over $k$, 
there is a closed immersion $j\colon X\hookrightarrow \A^N_k$ for some $N\in \mathbb Z$.
Moreover, since the tangent bundle $T_X$ is trivial, it follows that $N_{X/\A^N_k}$ is stably trivial, 
$N_{X/\A^N_k}\oplus \mathbf{1}^{N^\prime-N}_X\cong \mathbf{1}^{N^\prime-d}_X$, for $N^\prime=\bbZ$.
Consider the immersion given by the composite morphism 
\[X\to \A^N_k=\A^N_k\times 0\hookrightarrow \A^{N^\prime}_k,\] then $N_{X/\A^{N^\prime}_k}\cong \mathbf{1}^{N^\prime-d}_X$.
Redenote $N^\prime$ by $N$, then 
\begin{equation}\label{eq:NX/ANktriv}N_{X/\A^N_k}\cong \mathbf{1}^{N-d}_X.\end{equation}
Then there are regular functions $\varphi_i\in \calO_{\A^N_k}(\A^N_k)$, $i=1,\dots N-d$, such that
the differentials of $\varphi_i$ on $X$ induces the isomorphism \eqref{eq:NX/ANktriv}.
Hence $Z(\varphi_1,\dots, \varphi_{N-d})=X\amalg \hat X$.
So 
\begin{equation}\label{eq:(X-ANhatX-varphi)}(j, \A^N_k-\hat X, \varphi_1,\dots \varphi_{N-d})\end{equation} is a Zariski normal framing of $X$.

Let $\ovX\hookrightarrow \PP^N_k$ be the closed immersion given by the closure of $X$ in $\PP^N_k$.
Then $\dim X_\infty = \dim X-1=d-1$, where $X_\infty=\ovX\setminus X$.
Define \[\calX = X\times U, \ovcalX=\ovX\times U.\] Then \[\calX_\infty= X_\infty\times U,\]
and the condition of point (4) is satisfied.
Set \[\calE=X\times_Z U\subset \calX.\]
Since $X$ is $k$-smooth, and the morphism $r$ is smooth, $\calX,\calE\in \SmAff_U$, $\dim_U\calX=\dim_k X=d$, $\dim_U \calE = \dim r=1$.

Let $\Gamma,\Delta\subset \calE$
denote the graphs of morphisms 
$i\circ r,\can\colon U\to X$,
and $\gamma, \delta\colon U\to \calE$ be the corresponding morphism of $U$-schemes.
Then point (3) is done.
Denote by $\Delta_Z=\Delta\times_U Z\subset \calX\times_U Z$, then $\Delta_Z$ is the graph of the immersion $Z\to X$. 

Since the relative tangent bundle of the morphism $r\colon X\to Z$ is trivial, 
it follows that 
the relative tangent bundle of $\calE$ over $U$ is trivial.
Consequently, the normal bundle $N_{Z/X}$ is trivial,
and the normal bundle $N_{\Delta_Z/\calE}$ is trivial.
Point (2) is proven.

The base change of \eqref{eq:(X-ANhatX-varphi)}, and the canonical morphism $g\colon \calX\to X$
give the Zariski normal framing of $\calX$ over $U$. So point (1) is proven.

To finish the proof of point (0) we need to show that $N_{\calE/\calX}$ is trivial.
Consider the normal bundle $N_{Z/Z\times Z}$ with respect to the diagonal immersion $Z\hookrightarrow Z\times Z$. 
Since $\calE=X\times_Z U$, $\calX=X\times U$, it follow that 
the normal bundle $N_{\calE/\calX}$ equals 
the inverse image of $N_{Z/Z\times Z}$ with respect to the horizontal morphisms 
\[\xymatrix{
X\times U\ar[r]& Z\times Z\\
X\times_Z U \ar[u]\ar[r]& Z\ar[u].
}\] 
At the same time $N_{Z/Z\times Z}\cong T_Z$ is trivial by assumption.
Hence $N_{\calE/\calX}$ is trivial.

In addition let us note that $N_{\calX/\A^{N}_U}$ equals $N_{X/\A^N_k}$, hence it is trivial.
\end{proof}

\begin{lemma}\label{lm:r-RelcurveXhZ}
Suppose $i\colon Z\hookrightarrow X$ is a closed immersion of smooth affine schemes over a field $k$.
Let $U=X^h_Z$. 
Then there is \begin{itemize}
\item[(0)]
a framed $U$-scheme $\calC$ of relative dimension one; 
\item[(1)]
the closed subscheme $\Delta_Z\subset \calX_Z=\calX\times_U Z$ such that
the normal bundle $N_{\gamma_0(Z)/\calX_Z}$ is trivial;
\item[(2)]
$U$-morphisms $\gamma,\delta\colon U\to \calX$, and a $k$-morphism $g\colon \calX\to X$
such that 
$\gamma(Z)=\delta(Z)=\Delta_Z$, and
\begin{equation*}\label{eq:Conedimgammadeltacan}
g\circ \delta=\mathrm{can}\colon U\to X, g\circ \gamma=\mathrm{can}\circ r;    
\end{equation*}
\item[(3)]
a fine compactification of $\calC$ over $U$. 
\end{itemize}\end{lemma}
\begin{proof}
Firstly, applying Lemma \ref{lm:dimd} 
we get 
the framed $U$-scheme $\calX$ with the compactification $\ovcalX$, 
morphisms $\gamma,\delta\colon U\to \calX$,
such that \[\gamma(Z) = \delta(Z) =\Delta_Z\subset\calX\times_U Z,\] 
and closed subscheme $\calE\subset \calX$ that contains 
$\gamma(U)$ and $\delta(U)$ and such that 
\begin{equation}\label{eq:NDeltaZcalEZcalEcalX}
N_{\calE/\calX}\big|_{\Delta_Z}\cong \mathbf{1}^{d-1}_{\Delta_Z}, 
N_{\Delta_Z/\calE_Z}\cong \mathbf{1}_{\Delta_Z},\end{equation}
where $\calE_Z=\calE\times_U Z$.

By part (4) of Lemma \ref{lm:dimd} the scheme $X_\infty$ is a projective scheme of dimension $d-1$ over $k$.
Using Serre's theorem on ample bundles (see Prop. \ref{prop:CorofSerreTh}) 
for a large enough $b\in \mathbb Z$
we find sections $w_2\dots w_{d}\in \Gamma(X_\infty,\mathcal O(b) )$
such that $Z(w_2,\dots, w_{d})$ is finite (over $k$).
Denote \[C_\infty=Z(w_2,\dots ,w_d).\]

Denote by $I=I(\Delta_Z)\subset \calO_{X_Z}(X_Z)$ the vanishing ideal of the subscheme $\Delta_Z$,
and consider the closed subscheme $Z(I^2)\subset X_Z$. 
Since $\Delta_Z$ is finite over $Z$, $Z(I^2)$ is of such type too.
It follows by the first isomorphism in \eqref{eq:NDeltaZcalEZcalEcalX} that
there are regular functions
$y_i\in \calO_{Z(I^2)}(Z(I^2))$, $i=1,\dots ,d-1$,
such that $Z(y_1,\dots ,y_{d-1})=\calE\times_{\PP^N_U} Z(I^2)$.


By the Serre theorem, Prop. \ref{prop:CorofSerreTh}, again for some $l\in \mathbb Z$ 
there are sections
\begin{equation}\label{eq:s2mConditions}
s_2\dots s_{d}\in \Gamma(\calX,\mathcal O(lb) ),
s_i\big|_{\calX_\infty}=w_i^l, (s_i t_\infty^{-lb})\big|_{R}=y_i, s_i\big|_{\Gamma\cup \Delta_U}=0.
\end{equation}
where 
$\Gamma=\gamma(U)$, $\Delta_U=\delta(U)$.

Set
\[
\ovrelcurve=Z(s_2\dots s_{d})\cap \overline{\mathcal X}\subset \PP^N_U, %
\relcurve_\infty=\ovrelcurve\cap \calX_\infty, 
\relcurve=\ovrelcurve-\relcurve_\infty.
\]
Then by \eqref{eq:s2mConditions}
$\ovrelcurve\cap X_\infty=C_\infty\times U$ is finite over $U$,
and $\ovrelcurve$ is a projective variety over $U$ of the pure relative dimension one.
Applying Serre's theorem, Prop. \ref{prop:CorofSerreTh}, also one time we find $\zeta_c\in \Gamma(X_\infty,\mathcal O(d))$, for some $d\in \mathbb Z$, 
such that $\zeta_c$ is invertible on $C_\infty$.
Then
\[(\relcurve\hookrightarrow \ovrelcurve, \calO(d), \zeta_c, x_\infty^d)\] 
is a fine compactification of $\relcurve$ over $U$

Since $X$ is a framed $U$-scheme there is a canonical closed immersion
$X\hookrightarrow \A^N_U$, and $X$ is a clopen subscheme in the vanishing locus 
$Z(\psi_{d+1}, \dots \psi_N)$ for a set of regular functions $\psi_i$ on $\A^N_U$.
So $Z(\psi_{d+1}, \dots \psi_N)=X\amalg \hat X$.
Consider the functions $s_i/t_\infty^{bl}$ on $\calX$, and
let $\psi_2,\dots, \psi_d\in \calO_{\A^N_U}(\A^N_U)$ be lifts of $s_i/t_\infty^{bl}$.
Then 
\[\relcurve = Z(\psi_2,\dots, \psi_{N})-\hat{\calX}\subset \A^N_U.\]
Thus we have got 
the level $N$ framing 
of the affine curve $\relcurve$ of dimension one over $U$. 

It follows because of equalities \eqref{eq:s2mConditions} that the morphism $\relcurve\times_U Z \to Z$ is smooth over the subscheme $\Delta_Z$. Consequently, the morphism $\relcurve\to U$ is smooth over the subschemes $\Delta_U$ and $\Gamma$ in $\relcurve$.
\end{proof}

\begin{proposition}\label{prop:RigidityCodone}
Suppose $i\colon Z\hookrightarrow U$ is an affine smooth henselian pair over a field $k$, 
$\dim_k U=n$, $\codim_U Z=1$,
$T_Z$, and $N_{Z/U}$ are trivial.

Then 
the inverse image homomorphism induces the isomorphism $i^*\colon F(U)\simeq F(Z)$
for any $F$ as above.
\end{proposition}
\begin{proof}
Let $U=X^h_Z$ for a smooth affine scheme $X$ and a closed immersion $Z\hookrightarrow X$. 
It follows from Lemma \ref{lm:smoothretraction} that 
there is an \'etale neighbourhood $\widetilde X\in \SmAff_k$ with a closed immersion 
$\widetilde i\colon Z\to \widetilde X$ 
and a smooth morphism 
$\widetilde r\colon \widetilde X\to Z$
such that $U={\widetilde X}^h_Z$, and $\widetilde r\circ \widetilde i=\mathrm{id}_Z$.
Without loss of generality we can assume that $\widetilde X=X$.
Denote $\can\colon U\to X$, \[r= \widetilde r\circ \can\colon U\to Z, i=\can\circ i\colon Z\to U.\]

By 
Lemma \ref{lm:r-RelcurveXhZ} 
we get a framed $U$-scheme $\relcurve$ of relative dimension one, 
with a fine compactification, 
and morphisms
$\gamma,\delta\colon U\to \relcurve$, $g\colon \relcurve\to X$,
such that $g\circ \gamma=\widetilde i\circ r$, $g\circ \delta=\can$.
By Theorem \ref{th:Curve} it follows that for any $\sigma$-stable framed radditive presheaf $F$ 
the homomorphisms \[{can}^*, (\widetilde i\circ r)^*\colon F(X)\to F(U)\] are equal.
Since
$\widetilde i\circ r = \can \circ (i\circ r)$,
and since $U$ equals to the projectivel limit of all possible $X$ as above,
it follows that 
the homomorphisms \[{\id_U}^*, (i\circ r)^*\colon F(U)\to F(U)\] are equal.
\end{proof}


\subsection{General case}\label{subsect:codimonetrivnortangbunddlesreduct}





\begin{lemma}\label{lm:Thereiscovering}
Given a smooth affine henselian pair $Z\hookrightarrow U$ 
such that $Z$ and $U$ are schemes of pure dimension over a field $k$. 
Then there is a diagram 
\begin{equation}\label{diag:IdemSq}\xymatrix{
Z^\prime\ar[r]^{i^\prime}\ar@<0.5ex>[d]^{p_Z} &X^\prime\ar@<0.5ex>[d]^{p_X}\\
Z\ar[r]^i\ar@<0.5ex>[u]^{j^Z} &X\ar@<0.5ex>[u]^{j^X}
}\end{equation}
such that \begin{itemize}
\item[(1)] $i\colon Z\hookrightarrow X$, $i^\prime\colon Z^\prime\hookrightarrow X^\prime$ are closed embeddings of smooth affine schemes,
\item[(2)] $X^h_Z=U$; 
\item[(2)] $N_{Z^\prime/X^\prime}$, and $T_{Z^\prime}$ are trivial vector bundles (of constant rank) on $Z^\prime$; 
\item[(3)] both squares in \eqref{diag:IdemSq} are commutative, i.e. $p_Z\circ i=i^\prime\circ p_X$, $i^\prime\circ j_Z=j_Z\circ i$;
\item[(4)] $p_Z \circ j_Z\circ \id_Z$, $p_X \circ j_X\circ \id_X$.
\end{itemize}
\end{lemma}
\begin{proof}
By definition $U = X^h_Z$ for some $X\in \SmAff_k$. 
Moreover, 
by Lemma \ref{lm:smoothretraction} we can assume in addition that 
there is a retraction $r\colon X\to Z$.
Since $Z$ is affine there are vector bundles $T^\prime_Z$ and $N^\prime$ on $Z$ such that 
$T_Z\oplus \hat T_Z$ and $N_{Z/U}\oplus \hat N$ are trivial.
Let $\widetilde T=r^*(T_Z)$ and $\widetilde N=r^*(N_{Z/U})$. Then $\widetilde i^*(\widetilde T)=T^\prime$, $\widetilde i^*(\widetilde N)=N^\prime$.

Define an $Z$-scheme $Z^\prime=\widetilde T$ given by the total space of the vector bundle over $Z$, and similarly define an $X$-scheme $X^\prime=\widetilde T\oplus \widetilde N$.
Then $T_{Z^\prime}$ and $N_{Z^\prime/X^\prime}$ are equal to the inverse images of the vector bundles 
$T^\prime\oplus T_{Z}$ and $N^\prime\oplus N_{Z/X}$. 
Hence $T_{Z^\prime}$ and $N_{Z^\prime/U^\prime}$ are trivial.
\end{proof}

\begin{lemma}\label{lm:filtrcodone}
Let $Z\subset U$ be affine smooth henselian pair 
with $U$ of pure dimension over a field $k$; 
and assume the bundles $T_Z$ and $N_{Z/U}$ are trivial of constant rank on $Z$.
Then there is a sequence $Z=V_0\subset V_1\subset \dots\subset V_{n-1}\subset V_{n}=U$ of essentially smooth closed subschemes,
$\dim V_i = \dim Z + i$, $N_{V_i/V_{i+1}}\simeq O_{V_i}$.
\end{lemma}
\begin{proof}
Since $T_Z$ and $N_{Z/U}$, the bundle $T_U$ is trivial too.
Let $U^\prime$ be an affine scheme and $Z\hookrightarrow U^\prime$ be a closed immersion such that $U = (U^\prime)^h_Z$, and $T_{U^\prime}$ is trivial.
Then there is a closed immersion $U^\prime\subset \A^N_k$ such that $N_{U^\prime/\A^N_k}$ is trivial.
Let $t_1\dots t_{\dim Z}$ be the basis of $T_Z$, $t_{\dim Z+1}\dots t_{\dim U}$ be the basis of $N_{Z/U}$, 
$t_{\dim U+1}\dots t_{N}$ be the basis of
$N_{U_i/\A^N_k}$.
Let $\hat t_i$ denotes the dual basis of $\Omega(\A^N_k)\big|_{Z}$.
Choose a lift $f_i\in \calO[\A^N_k]$ of $\hat t_i$. 
The desired filtration is given by $Z(f_1\dots f_i)\times_{\A^{N}_k} U = V_{n-i}$.
\end{proof}

\begin{theorem}\label{th:Rigidity}
Suppose $Z\hookrightarrow U$ is an affine smooth henselian pair over a field $k$, 
and $F\colon Sm_k\to Ab$ is an $l_\varepsilon$-torsion homotopy invariant $\sigma$-stable linear framed presheaf for some $l\in \mathbb Z$, $l\in k^\times$.
Then the inverse image homomorphism induces the isomorphism $F(U)\simeq F(Z)$.
\end{theorem}
\begin{proof} 
Without loss of generality we can assume that $U$ is connected.
If $U=U_1\amalg U_2$ and the claim holds for the henselian pairs $Z\cap U_i\hookrightarrow U_i$, $i=1,2$, then the claim follows for the pair $Z\hookrightarrow U$ as well.

Since $Z\to U$ is a henselian pair, if $U$ is connected, then $Z$ is connected.
Since $U$ and $Z$ are smooth, $Z$ and $U$ are locally of constant dimension. So $T_Z$ and $N_{Z/U}$ are vector bundles of constant rank on $Z$.

By Lemma \ref{lm:Thereiscovering} 
the henselian pair $Z\to U$
is a retract of a henselian pair
$Z^\prime\to U^\prime$
\begin{equation}\label{diag:RetractMorUsing}\xymatrix{
Z^\prime\ar[r]^{i^\prime}\ar@<0.5ex>[d]^{p_Z} &U^\prime\ar@<0.5ex>[d]^{p_U}\\
Z\ar[r]^i\ar@<0.5ex>[u]^{j^Z} &U\ar@<0.5ex>[u]^{j^U}.
}\end{equation}
such that $N_{Z^\prime/X^\prime}$, and $T_{Z^\prime}$ are trivial vector bundles of constant rank on $Z^\prime$.
%
Then by Lemma \ref{lm:filtrcodone} there is a filtration 
\[Z^\prime=U^\prime_0\hookrightarrow U^\prime_1\hookrightarrow \dots \hookrightarrow U^\prime_{n-1}\hookrightarrow U^\prime_n=U^\prime,\]
$\codim_{U^\prime_{i+1}} U^\prime_i=1$, $N_{U^\prime_{i+1}/U^\prime_i}\simeq \mathbf{1}_{U^\prime_i}$.

Now by Proposition \ref{prop:RigidityCodone} 
\[F(Z^\prime)=F(U^\prime_1)=\dots =F(U^\prime_{i-1})=F(U^\prime).\]
Since the morphism $i$ is a retract of $i^\prime$, see \eqref{diag:RetractMorUsing}, the morphism \[i^*\colon F(U)\to F(Z)\] is a direct summand of the morphism \[(i^\prime)^*\colon F(U^\prime)\to F(Z^\prime).\]
Whence since $(i^\prime)^*$ is an isomorphism by the above, it follows that $i^*$ is an isomorphism.
\end{proof}

\subsection{Corollaries.}

\begin{corollary}
Let $k$ be a field and $l\in \mathbb Z$, $(l,\chark k)=1$.
Then for an $l$-torsion homotopy invariant $\sigma$-stable linear framed presheaf $F$ and an affine smooth henselian pair $Z\subset U$ over $k$
we have $F(U)\simeq F(Z)$. 
\end{corollary}
\begin{proof}
Let $2\in k^\times$. The presheaf $F$ is $2l_\varepsilon$-torsion presheaf, since $2l_\varepsilon=2h$, and $(2l,\chark k)=1$.
Let $2=0\in k$. Then $F$ is $l_\varepsilon$-torsion, since $l_\varepsilon=l$.
Hence in both cases $F(U)\simeq F(Z)$ by Theorem \ref{th:Rigidity}.
\end{proof}
\begin{corollary}\label{cor:phiGW}
Let $i\colon Z\hookrightarrow U$ be a smooth affine henselian pair over a field $k$.
Let $E\colon \SH(k)\to Ab^\mathrm{op}$ be an additive presehaf (of abelian gorups) on the stable motivic homotopy category $\SH(k)$. 
Consider the inverse image homomorphism
$i^*\colon E(U)\to E(Z)$ is an isomorphism,
where $E(Z)$ and $E(U)$ is the value of the presheaf on $\Sm_k$ induced by $E$, and $E(U)$ is the stalk on $U$.


Then if $\phi E=0$ for some 
$\phi\in \GW(k)$, 
under the meaning of the homomorphism $\GW(k)\to [\mathbb S,\mathbb S]_{\SH(k)}$, 
such that $\rank \phi$ is invertible in $k$,
then $i^*$ is an isomorphism.
\end{corollary}
\begin{proof}
The case of the presheaf $E$ such that $l_\varepsilon E$ for some $l\in \mathbb Z$, $l\in k^\times$,  
follows from Theorem \ref{th:Rigidity} in view of Lemma \ref{lm:EovZFSHandstable} 
and Remark \ref{rem:l-epsilonTorsionEinSH}.
%

Consider the case $\phi E=0$ for a quadratic space $\phi$ as in the corollary,
and $\chark  k\neq 2$.
Let $h =\langle 1\rangle +\langle -1\rangle\in \GW(k)$ be the hyperbolic form of rank 2, and define $l_\varepsilon\in \GW(k)$ as $\sum_{i=0}^l \langle (-1)^i\rangle$. 
One has $\phi h = lh=(2l)_\varepsilon$ for $l = \rank \phi$, hence $(2l)_\varepsilon E = \phi hE = 0$. 
Since $\chark  k\neq 2$, the claim follows by the above.

Let $\chark k = 2$. Then $h=2$, and $l_{\varepsilon}=l$, for any $l\in \mathbb Z$. 
Recall that in the Witt ring $W(k)$ every element can be represented by a diagonal form whence for every element in $\Witt(k)$ its square is either 0 or 1. 
Since the rank of $\phi$ is odd, we have $\phi^2=1$ in $\Witt(k)$. 
Then $\phi^2=1+2m$  in $\GW(k)$, where $m= (\rank\phi^2-1)/2$, and so $\phi^2=l_\varepsilon^2$ with $l=\rank\phi$.
Now the claim follows, since $l^2_\varepsilon E = \phi^2E = 0$.
\end{proof}

\begin{remark}
The suggestion to replace the $n$-torsion condition, $n\in\mathbb Z$, 
by torsion with respect to some class of 
elements in $[\mathbb S,\mathbb S]_{\SH(k)}$, 
and expection that rigidity results should hold for 
such presheaves
was made originally by A.~Ananievskiy.
The first version of \cite{AD-Rigid} concerned $n$-torsion and $n_\epsilon$-torsion presheaves.

The argument how to prove 
the result of \cite{AD-Rigid} for $\phi$-torsion for an arbitrary quadratic form $\phi\in \GW(k)$
was founded firstly by T.~Bachamnn
And then A.~Ananievskiy had shown that this case is actually equivalent to the $n_\epsilon$-torsion case.
In the article we use this last short argument.
\end{remark}

\appendix

\section{Appendix: Serre's theorem on ample line bundles}
In the subsection we write a consequence of the Serre theorem on ample bundles.
\begin{proposition}(A particular case of Serre's theorem on ample bundles)\label{prop:CorofSerreTh}
Let $\calO(1)$ be an ample invertible sheaf on a scheme $X$, and $\calL$ be an invertible sheaf on $X$.
Then for any $Z\subset X$ closed subscheme then there is $N\in \mathbb Z$
the restriction homomorphism is surjective $\Gamma(X,\calL(l))\twoheadrightarrow \Gamma(Z,\calL(l))$, $\forall l>N$.
Here $\calL(l)=\calL \otimes\calO(l)$.
\end{proposition}
\begin{proof}
Consider the short exact sequence of coherent sheaves on $X$
\begin{equation}\label{eq:restrcohshILiiL}
0\to \mathcal I \to \calL\xrightarrow{r} i_* i^*(\calL)\to 0, 
\end{equation} where $i\colon Z\hookrightarrow X$.
By Serre's theorem \cite[theorem 5.2]{Ha77} it follows that $H^1(X,\mathcal I(l))=0$ for all $l$ large enough.
Then the long exact sequence of homologies for \eqref{eq:restrcohshILiiL} gives the surjection
\[\Gamma(X,\calL(l))\twoheadrightarrow \Gamma(X,i_* i^* \calL(l)) = \Gamma(Z,\calL(l)).\]
\end{proof}


\section{Appendix: Matrix homotopies}

Let $Z\hookrightarrow X$ be a closed immersion of affine schemes.
In this section we prove that 
if a matrix $C$ over $\calO_X(X)$ 
equals to the identity on $Z$, 
then $C$ is an elementary matrix on $X^h_Z$.

Secondly, we write the proof 
for the fact that any elementary matrix
on an affine scheme 
is $\A^1$-homotopy equivalent to the identity.

\subsection{Elementary transformations, and classes of isomorphisms.}

Let $X$ be an affine scheme.
For any vector bundles $N$ and $N^\prime$ on $X$ we 
denote by $\Hom(N,N^\prime)$ the set of homomorphisms, 
denote by $\Iso(M,N^\prime)$ the set of isomorphisms.
$\Aut(N)=\Iso(N,N)$ is the group of automorphisms.
In particular, $\GL_n(X)=\Iso(\mathbf 1^n_X, \mathbf 1^n_X)$, where $\mathbf 1^n_X$ denote the free vector bundle of the rank $n$.

Let $w\colon N\cong N_1\oplus N_2$.  
Then every $C\in \Hom(N_2,N_1)$ defines an automorphism $T\in\Aut(N)$ given by the matrix
\begin{equation}\label{eq:T(idCid)}
T=e_{w}(C)=\big(
\begin{smallmatrix} \id_{N_1} & C\\ 0& \id_{N_2} \end{smallmatrix}
\big).
\end{equation}
\begin{definition}
For any vector bundle $N$ on $X$ 
we denote by $\El_N(X)$ the subgroup of $\Iso(N,N)$ 
generated by the elements $T$ as above for all splitings? $N\cong N_1\oplus N_2$.
\end{definition}

In particular, if $N=\mathbf{1}^n_X$ the group $\El_N(X)=\El_n(X)$ is generated by the elementary transformations.
\begin{definition}
An elementary transformation $e_{i,j}(\lambda)\in \GL_n(X)$, for $i,j=1,\dots ,n$, $\lambda\in \calO_X(X)$ is an automorphism of $\mathbf 1^n_X$ 
defined by the action on the basis elements $t_1$:
\[t_l\mapsto t_l, i\neq i, t_i\mapsto t_i+\lambda t_j.\]
\end{definition}

\begin{definition}\label{def:[N,Nprime]E}
Denote by $[N, N^\prime]_E$ the set of classes in $\Iso(N, N^\prime)$ with respect to the action of $\El_N(X)$. 
\end{definition}

Let 
\begin{equation}\label{eq:NN1N2exactsequence}N_1\hookrightarrow N\twoheadrightarrow N_2\end{equation}
be a short exact sequence of vector bundles on an affine scheme $X$.
Let 
\[\mu\colon N_2\simeq \mathbf{1}^{n_1}_X, \nu\colon N_1\simeq \mathbf{1}^{n_2}_X\]
be isomorphisms.
Since $X$ is affine there is a splitting
\begin{equation}\label{eq:NN1N2splitting}N\cong N_1\oplus N_2\end{equation}
\begin{definition}\label{def:[mu,nu]}
Denote by 
\begin{equation}\label{eq:[mu,nu]}[\mu,nu]\in [N,\mathbf{1}^n_X]_E,\end{equation} 
for $n=n_1+n_2$, the class of the isomorphism 
\[\tau=mu\oplus \nu\colon N\simeq \mathbf{1}^n_X.\] 
%
\end{definition}

\begin{lemma}\label{lm:[mu,nu]}
The class \eqref{eq:[mu,nu]} is independent of the choice of the splitting \eqref{eq:NN1N2splitting}.
\end{lemma}
\begin{proof}
Fix one splitting $w_1\colon N= N_1\oplus N_2$, and let 
$w_2\colon  N\cong N_1\oplus N_2$ be another one agreed with \eqref{eq:NN1N2exactsequence}.
Then the automorphism $w_2^{-1} w_1\in \GL_N(X)$ 
in view of the splitting $w_1$
is given by a matrix of the form \eqref{eq:T(idCid)} for some $C\in \Hom(N_2,N_1)$.
So the claim follows.
\end{proof}

\begin{lemma}\label{lm:diagSLinEn}
Let $R$ be a ring, and $B=\langle \lambda_1,\dots \lambda_n\rangle\in \SL_n(R)$ be a diagonal matrix, $\lambda_1\cdots\lambda_n=1$.
Then $B$ is an elementary matrix.
\end{lemma}
\begin{proof}
Any matrix as in the lemma is a product of matrices of the form 
\[m_{i,j}(\lambda)=\langle 1,\dots 1,\lambda, 1,\dots,1,\lambda^{-1}, 1,\dots 1\rangle\] 
with $\lambda$ and $\lambda^{-1}$ at $i$-th and $j$-th positions.
Then we using the equality
$m_{i,j}(\lambda) = E_{j,i}(\lambda^{-1}-1) E_{i,j}(-1) E_{j,i}(\lambda-1) E_{i,j}(-\lambda^{-1})$
we get the claim.
\end{proof}

\subsection{Elementary matrices on henselian schemes}

\begin{lemma}\label{lm:ElemHomdiagdetectinvrad}
Let $C\in \mathrm{GL}_{n\times n}(R)$ for a ring $R$.
Suppose that 
\[\begin{cases}
c_{i,i}\in R^\times, i=1,\dots n,\\ 
c_{i,j}\in J(R), i\neq j,
\end{cases}\]
the coefficients $c_{i,i}\in R^\times$ are invertible in $R$ for $i=1,\dots n$, and 
$c_{i,j}\in J(R)$ are in the Jacobson radical, $i\neq j$.
Then $C\sim (\Jac(C),1,\dots 1)$ denotes the Jacobian.
\end{lemma}
\begin{proof}
We prove the claim by induction on $n$. 
To do this we modify $C$ by elementary transformations 
to get the diagonal matrix 
$(C^\prime, 1)$, for $C^\prime\in \GL_{n-1}(X)$ that satisfy the same properties as in the lemma. 

Denote $C_0=C$ and then the required equivalence is provided by the sequence
\[\begin{array}{lll}
C\sim C_1 &=& C_0\circ e_{n-1,n}(c_{n-1,n}c_{n,n}^{-1})\circ \dots \circ e_{1,n}(c_{1,n}c_{n,n}^{-1}),\\
C\sim C_2 &=& e_{n,1}(c_{1,n}c_{n,n}^{-1})\circ \dots \circ e_{n,n-1}(c_{n,n-1}c_{n,n}^{-1}) \circ C_1,\\
C\sim C_3 &=& C_2\circ e_{n,n-1}(1-c_{n,n}),\\
C\sim C_4 &=& C_3\circ e_{n-1,n}(1).
\end{array}\]
\end{proof}

\begin{corollary}\label{cor:henspairElemHomdiagdetectinvrad}
Let $i\colon Z\to X$ be a closed immersion that defines an affine henselian pair.
Let 
$C\in \mathrm{GL}_{n\times n}(X)$ be such that
$i^*(C)=\Id_n$ is identity.
Then $C$ is elementary, i.e. $C\in \El_n(X)$.
\end{corollary}
\begin{proof}
Due to Lemma \ref{lm:ElemHomdiagdetectinvrad} the claim follows since the vanishing ideal of $Z$ in $\mathcal O_X(X)$ is contained in the Jacobson radical. Actually for any $v\in \calO_X(X)$ such that $i^*(v)=0$ the function $1+v$ is invertible on $X$.
\end{proof}


\subsection{$\A^1$-homotopies for elementary matrices}
For any vector bundle $N$ on a scheme $X$ denote by $\mathbb Z\GL_N(X)$ the free abelian group generated by the set of elements of $\GL_N(X)$.

\begin{lemma}\label{lm:ElementaryMatHomotopy}
Let $E\in \mathrm \El_N(X)$ be an elementary matrix for some vector bundle $N$ on a scheme $X$.
Then there is 
$H\in \mathbb Z\GL_N(X\times\A^1)$, 
such that 
$i^*_0(H)=\id_N$, $i^*_1(H)=E$, 
where 
$i_0,i_1\colon X\to X\times\A^1$ 
denote the zero and unit sections.
\end{lemma}
\begin{proof}
Consider a matrix $T=e_{w}(\lambda C)$ 
of the form \eqref{eq:T(idCid)} 
for some splitting $w\colon N\cong N_1\oplus N_2$.
Then the matrix 
\[e_{w}(\lambda C)=\big(\begin{smallmatrix} \id_{N_1} & \lambda C\\ 0& \id_{N_2}\end{smallmatrix}\big)\]
gives the required homotopy.

In general, for a matrix 
$T=e_{w_l}(C_l)\dots e_{w_0}(C_0)$ given by the product of matrices of the form \eqref{eq:T(idCid)}
the required homotopy is given by
\[\sum_{i=0}^l e_{w_i}(\lambda C_i) T_i,\] where $T_i=e_{w_i}(C_i)\dots e_{w_0}(C_0)$.
\end{proof}


%
%
%


\end{document}